%% file: optquant.tex
\documentclass{article}
\usepackage[nonatbib,preprint]{neurips_2021}

\usepackage{optquant}

\begin{document}

% If your sepaper is accepted and the title of your paper is very long,


%\runningtitle{Convergence bounds for Wasserstein approximation}
\title{Non-asymptotic convergence bounds for\\
  Wasserstein approximation using point clouds}

 \author{Quentin M\'erigot \\ 
   Université Paris-Saclay, CNRS,\\ Laboratoire de mathématiques d’Orsay, \\
   91405, Orsay, France\\
   Institut Universitaire de France 
   \And
   Filippo Santambrogio \\ Univ Lyon,\\ Université Claude Bernard Lyon 1, CNRS \\
   UMR 5208, Institut Camille Jordan\\  F-69622 Villeurbanne\\
      Institut Universitaire de France 
  \And Cl\'ement Sarrazin \\ 
Université Paris-Saclay, CNRS,\\ Laboratoire de mathématiques d’Orsay\\ 91405, Orsay, France }

\maketitle

\begin{abstract}
	%% In several problems in machine learning and inverse problems,
        %% one is tasked with generating discrete data, as if sampled
        %% from a model distribution $\rho$. A common way to do so is to
        %% chose a discrete uniform distribution of points, minimizing
        %% the Wasserstein distance to the measure $\rho$. This problem
        %% is non-convex and the fact that, in experiments, following a
        %% slightly adjusted version of Lloyd algorithm leads to already
        %% good quasi-minimizers is quite surprising.  We provide
        %% explicit upper bounds for the convergence speed of this Lloyd
        %% algorithm, starting from a cloud of points far from one
        %% another, and this already after one step. Similar bounds can
        %% be deduced, from this one step case, for a corresponding
        %% gradient descent with general step-length. These bounds can
        %% very naturally be linked to a type of Poliak-\L ojasiewicz
        %% inequality for the minimized transport cost, but featuring an
        %% extra term which penalizes small distances between Dirac
        %% masses in the discrete distribution.

        %% \bigskip
        
Several issues in machine learning and inverse problems require to
generate discrete data, as if sampled from a model probability
distribution. A common way to do so relies on the construction of a
uniform probability distribution over a set of $N$ points which
minimizes the Wasserstein distance to the model distribution. This
minimization problem, where the unknowns are the positions of the
atoms, is non-convex. Yet, in most cases, a suitably adjusted version
of Lloyd's algorithm --- in which Voronoi cells are replaced by Power
cells --- leads to configurations with small Wasserstein error. This
is surprising because, again, of the non-convex nature of the problem,
as well as the existence of spurious critical points. We provide
explicit upper bounds for the convergence speed of this Lloyd-type
algorithm, starting from a cloud of points sufficiently far from each
other. This already works after one step of the iteration procedure,
and similar bounds can be deduced, for the corresponding gradient
descent. These bounds naturally lead to a modified Poliak-Łojasiewicz
inequality for the Wasserstein distance cost, with an error term
depending on the distances between Dirac masses in the discrete
distribution.

%We consider the gradient flow, in the class
%of uniform probability measures over N points (samplings), of the
%quadratic Wasserstein distance to the target: after computing the
%optimal power diagram, every point moves towards the barycenter of its
%own cell. Using a geometric result on the barycenters of Power
%power cells in terms of the minimal distance between their nuclei,
%together with a suitable Gronwall Lemma, we estimate the distance at
%time $t$ between the gradient flow and the target. Optimizing our
%parameters, we show that, running the gradient flow for a duration
%which logarithmically depends on the data, a suitable estimate can be
%obtained in terms of the initial distribution of points. These bounds
%are compared to numerical experiments.
\end{abstract}

%% \subsubsection*{Acknowledgements}
%% All acknowledgments go at the end of the paper, including thanks to
%% reviewers who gave useful comments, to colleagues who contributed to
%% the ideas, and to funding agencies and corporate sponsors that
%% provided financial support.  To preserve the anonymity, please include
%% acknowledgments \emph{only} in the camera-ready papers.

\input{sec1-intro}
\input{sec2-deterministic}

\input{sec3-gradientflow}

\input{sec4-numeric}

\section{Discussion}\label{sec:discussion}
We have studied the problem of minimizing the Wasserstein distance
between a fixed probability measure $\rho$ and a uniform measure over
$N$ points $\delta_Y$, parametrized by the position of the points $Y =
(y_1,\hdots,y_N)$. The main difficulty is the nonconvexity of the
Wasserstein distance $F_N: Y\in(\Rsp^d)^N
\mapsto\frac{1}{2}\Wass_2^2(\rho,\delta_Y)$, which we tackled by
means of a modified Polyak-\L{}ojaciewicz inequality \eqref{eq:PL}. One
limitation of our work is that the terms replacing $\min F_N$ in the
Polyak-\L{}ojaciewicz inequality \eqref{eq:PL} does not match the
theoretical bounds recalled in \cref{rmk:rateminF}. Future work will
concentrate on bridging that gap, but also on deriving consequences
for the algorithmic resolution of Wasserstein regression problems
$\min_{\theta} \Wass_2^2(\rho,T_{\theta\#}\mu),$ starting with the
case where $\theta \mapsto T_\theta$ is linear.

\begin{ack} This work was supported by a grant from the French ANR (MAGA, ANR-16-CE40-0014).
\end{ack}

\clearpage
\bibliographystyle{plain}
\bibliography{optquant}

\clearpage
\appendix
\input{appendix}

\end{document}

%% file: sec1-intro.tex
\section{Introduction}

In recent years, the theory of optimal transport has been the source
of stimulating ideas in machine learning and in inverse
problems. Optimal transport can be used to define distances, called
Wasserstein or earth-mover distances, between probability
distributions over a metric space. These distances allows one to
measure the closeness between a generated distribution and a model
distribution, and they have been used with success as data attachment
terms in inverse problems.  Practically, it has been observed for
several different inverse problems that replacing usual loss functions
with Wasserstein distances tend to increase the basin of convergence
of the methods towards a good solution of the problem, or even to
convexify the landscape of the minimized energy
\cite{feydy2017optimal,engquist2016optimal}. This good behaviour is
not fully understood, but one may attribute it partly to the fact that
the Wasserstein distances  encodes the geometry of the underlying
space. A notable use of Wasserstein distances in machine learning is
in the field of generative adversarial networks, where one seeks
to design a neural network able to produce random examples whose
distribution is close to a prescribed model distribution
\cite{arjovsky2017wasserstein}.

\paragraph{Wasserstein distance and Wasserstein regression}
Given two probability distributions $\rho,\mu$ on $\Rsp^d$, the
Wasserstein distance of exponent $p$ between $\rho$ and $\mu$ is a way
to measure the total cost of moving mass distribution described by
$\rho$ to $\mu$, knowing that moving a unit mass from $x$ to $y$
costs $\nr{x-y}^p$. Formally, it is defined as the value of an optimal
transport problem between $\rho$ and $\mu$:
\begin{equation}\label{eq:wass}
  W_p(\rho,\mu) = \left(\min_{\pi \in \Pi(\rho,\mu)} \int \nr{x -
    y}^p \dd \pi(x,y)\right)^{1/p},
\end{equation}
where we minimize of the set $\Pi(\rho,\mu)$ of \emph{transport plans}
between $\rho$ and $\mu$, i.e. probability distributions over
$\Rsp^d\times \Rsp^d$ with marginals $\rho$ and $\mu$.  Standard
references on the theory of optimal transport include books by Villani
and by Santambrogio
\cite{villani2003topics,villani2008optimal,santambrogio2015optimal},
while the computational and statistical aspects are discussed in a
survey of Cuturi and Peyré \cite{peyre2019computational}.

In this article, we consider regression problems with respect to the
Wasserstein metric, which can be put in the following form
\begin{equation}\label{eq:WGAN}
  \min_{\theta \in\Theta} \Wass_p^p(T_{\theta\#} \mu, \rho),
\end{equation}
where $\mu$ is the \emph{reference distribution}, a probability
measure on $[0,1]^\ell$, $\rho$ is the \emph{model distribution}, a
probability measure on $\Rsp^d$, and where $T_\theta: [0,1]^\ell\to
\Rsp^d$ is a family of maps indexed by a parameter $\theta\in\Theta$.
In the previous formula, we also denoted $T_{\theta}\# \mu$ the image
of the measure $\mu$ under the map $T_\theta$, also called
\emph{pushforward} of $\mu$ under $T_\theta$. This image measure is
defined by $T_{\theta\#} \mu(B) := \mu(T_\theta^{-1}(B))$ for any
measurable set $B$ in the codomain of $T_\theta$.  In this work, we
will concentrate on the quadratic Wasserstein distance $W_2$.  Several
problems related to the design of generative models can be put under
the form \eqref{eq:WGAN}, see for instance
\cite{genevay2018learning,arjovsky2017wasserstein}.  Solving
\eqref{eq:WGAN} numerically is challenging for several reasons, but in
this article we will concentrate on one of them: the non-convexity of
the Wasserstein distance under displacement of the measures.

\paragraph{Non-convexity of the Wasserstein distance under displacements.} 
It is well known that the Wasserstein distance is convex for the
standard (linear) structure of the space of probability measures,
meaning that if $\nu_0$ and $\nu_1$ are two probability measures and
$\nu_t = (1-t)\nu_0 + t\nu_1$, then the map $t\in[0,1]\mapsto
\Wass_p^p(\nu_t,\rho)$ is convex. Using a terminology from physics, we
may say that the Wasserstein distance is convex for the
\emph{Eulerian} structure of the space of probability measures,
e.g. when one interpolates linearly between the densities. However, in
the regression problem \eqref{eq:WGAN}, the perturbations are \emph{Lagrangian} rather
than Eulerian, in the sense that modifications of the parameter
$\theta$ leads to a displacement of the support of the measure
$T_{\theta\#} \mu$. This appears very clearly in particular when $\mu$
is the uniform measure over a set $X = (x_1,\hdots,x_N)$ of $N$ point
in $[0,1]^d$, i.e. $\mu = \delta_X$ with
$$ \delta_X \eqdef \frac{1}{N}\sum_{i=1}^N\delta_{x_i}.$$
In this case $T_{\theta\#} \mu$ is the uniform measure
over the set $T_\theta(X) = (T_\theta(x_1),\hdots,T_\theta(x_n))$,
i.e. $T_{\theta\#} \mu = \delta_{T_\theta(X)}$. In this article, we
will therefore be interested by the function
\begin{equation}
  \label{eq:F} \F_N: Y\in(\R^d)^N \mapsto \frac{1}{2} W^2_2(\rho,\delta_Y).
%  \footnote{\textcolor{red}{P-e plutôt $F_N:Y\in(\R^d)^N\mapsto\frac{1}{2} W^2_2(\rho,\delta_Y)$ pour que le N apparaisse dans la définition?(Clément)}}
\end{equation}

This function $F_N$ is \emph{not} convex, and actually exhibits
(semi-)concavity properties. This has been observed first in
\cite{ambrosio2008gradient} (Theorem 7.3.2), and is related to the positive curvature of the 
Wasserstein space. A precise statement in the
context considered here may also be found as Proposition~21
in~\cite{merigot2016minimal}. A practical consequence of the lack of
convexity of $F_N$ is that critical points of $F_N$ are not
necessarily global minimizers. It is actually easy to construct
examples families of critical points $Y_N$ of $F_N$ such that
$F_N(Y_N)$ is bounded from below by a positive constant, while $\lim_{N\to\infty}\min F_N=0$, so that the ratio between $F_N(Y_N)$ and $\min F_N$ is arbitrarily large as $N\to +\infty$. This can be done by
concentrating the points $Y_N$ on lower-dimensional subspaces of
$\Rsp^d$, as in Remarks~\ref{rmk:badcritical} and \ref{rmk:optimassump}.

When applying gradient descent to the nonconvex optimization problem
\eqref{eq:WGAN}, it is in principle possible to end up on local
minima corresponding to a high energy critical points of the
Wasserstein distance, regardless of the non-linearity of the map
$\theta \mapsto T_{\theta\# \sigma}$. Our main theorem, or rather its
Corollary~\ref{cor:PL} shows that if the points of $Y$ are at distance at least
$\eps>0$ from one another, then
\begin{equation*}
  F_N(Y) - C \frac{1}{N\eps^{d-1}} \leq N \nr{\nabla
    F_N(Y)}^2.
\end{equation*}
In the previous equation $ \nr{\nabla F_N(Y)}$ denotes the Euclidean
norm of the vector in $\R^{Nd}$ obtained by putting one after the
other the gradients of $F_N$ w.r.t. the positions of the atoms
$y_i$. We note that due to the weights $1/N$ in the atomic measure
$\delta_Y$, the components of this vector are in general of the order
of $1/N$, see Proposition~\ref{prop:gradF}. This inequality resembles
the Polyak-\L ojasiewicz inequality, and shows in particular that if
the quantization error $F_N(Y) = \Wass_2^2(\rho,\delta_Y)$ is large,
i.e. larger than $\eps^{1-d}/N$, then the point cloud $Y$ is not
critical for $F_N$. From this, we deduce in \cref{thm:gf} that if the
points in the initial cloud are not too close to each other at the
initialization, then the iterates of fixed step gradient descent
converge to points with low energy $F_N$, despite the non-convexity of
$F_N$.
\begin{figure}[t]                                                                                
  \centering
  \begin{tabular}{m{.22\textwidth}m{.22\textwidth}m{.22\textwidth}m{.22\textwidth}}
        \includegraphics[width=0.22\textwidth]{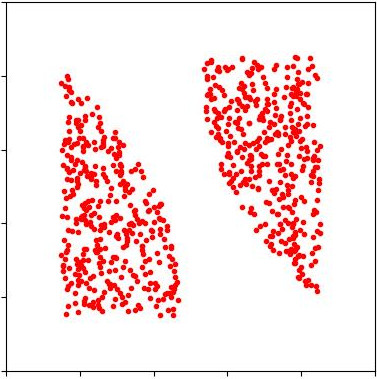}&
        \includegraphics[width=0.215\textwidth]{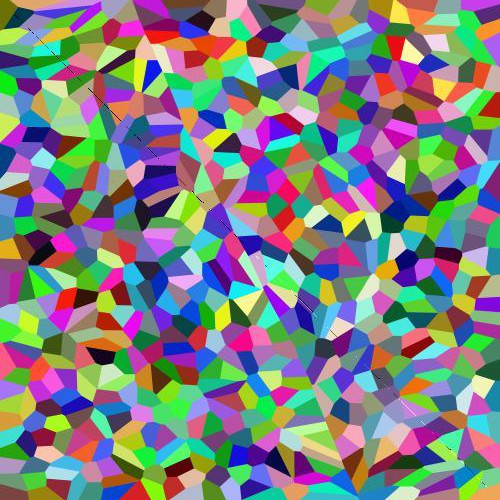}&
        \includegraphics[width=0.22\textwidth]{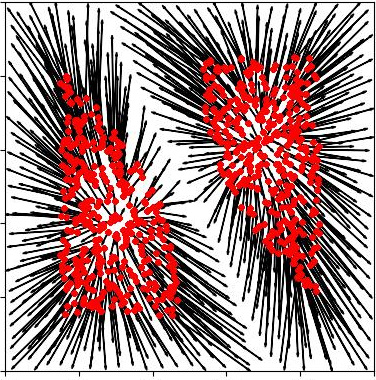}&
        \includegraphics[width=0.22\textwidth]{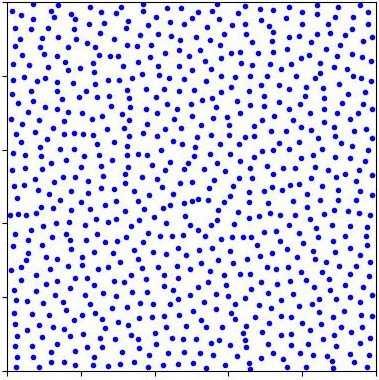}
  \end{tabular}
  \caption{\label{fig:lloyd}From left to right, a point cloud $Y^0$ in the square $\Omega=[0;1]\times[0;1]$, the associated power cells $P_i(Y)$ in the optimal transport to the Lebesgue measure on $\Omega$, the vectors $-N \nabla F_N(Y^0) = B_N(Y^0)-Y^0$ followed during the Lloyd step and the positions of the barycenters $Y^1 = B_N(Y)$.}
\end{figure}

\paragraph{Relation to optimal quantization}
Our main result also has implications in terms of the \emph{uniform
  optimal quantization problem}, where one seeks a point cloud $Y =
(y_1,\hdots,y_N)$ in $(\Rsp^d)^N$ such that the uniform measure
supported over $Y$, denoted $\delta_Y$, is as close as possible to the
model distribution $\rho$ with respect to the $2$-Wasserstein
distance:
\begin{equation} 
  \label{pb:main} \min_{Y \in \Omega^N} F_N(Y).
\end{equation}
%% where the quadratic Wasserstein distance is given by
%% $$ W^2_2(\rho,\delta_Y) = \min \left\{ \int \nr{x - T(x)}^2 \dd\rho
%% \mid T_\#\rho = \delta_Y \right\}. $$ In the previous formula, the
%% condition $T_\# \rho = \delta_Y$ on the map $T:\Omega\to Y$ means that
%% $\delta_Y$ is the image of the measure $\rho$ under the map $T$, in
%% which case we say that $T$ \emph{transports} $\rho$ onto
%% $\delta_Y$. Since $Y$ is finite, the condition $T_\#\rho = \mu$ is
%% equivalent to
%% $$ \forall i\in\{1,\hdots,N\},~~\rho(T^{-1}(\{y_i\})) = \frac{1}{N}. $$
The uniform optimal quantization problem \eqref{pb:main} is a very
natural variant of the (standard) \emph{optimal quantization problem},
where one does not impose that the measure supported on $Y$ is uniform:
\begin{equation}
  \label{pb:standard} %\min_{\mu\in\mathcal{P}_N(\Rsp^d)} \frac{1}{2} W^2_2(\rho, \mu) =
  \min_{Y \in (\Rsp^d)^N} G_N(Y),~~ \hbox{ where } G_N(Y) = \min_{\mu\in\Delta_N} \Wass^2_2(\rho,\sum_{i=1}^N \mu_i \delta_{y_i}),
\end{equation}
and where $\Delta_N \subseteq \Rsp^N$ is the probability simplex.  This
standard optimal quantization problem is a cornerstone of sampling
theory, and we refer the reader to the book of Graf and Luschgy
\cite{graf2007foundations} and to the survey by Pagès
\cite{pages2015introduction}. The uniform quantization problem
\eqref{pb:main} is less common, but also very natural.  It has been
used in imaging to produce stipplings of an image
\cite{de2012blue,balzer2009capacity} or for meshing purposes
\cite{goes2014weighted}.  A common difficulty for solving
\eqref{pb:standard} and \eqref{pb:main} numerically is that the
minimized functionals $F_N$ and $G_N$ are non-convex and  have
many critical points with high  energy. However, in practice, simple
fixed-point or gradient descent stategies behave well when the initial
point cloud is not chosen adversely. Our second contribution is a
quantitative explanation for this good behaviour in the case of the
uniform optimal quantization problem.

%\paragraph{Lloyd's algorithm}
Lloyd's algorithm \cite{lloyd1982least} is a fixed point algorithm for
solving approximately the standard optimal quantization problem
\eqref{pb:standard}. Starting from a point cloud $Y^k =
(y^k_1,\hdots,y^k_N) \in (\Rsp^{d})^N$ with distinct points, one
defines the next iterate $Y^{k+1}$ in two steps. First, one computes
the Voronoi diagram of $Y$, a tessellation of the space into convex
polyhedra $(V_i(Y^k))_{1\leq i\leq N}$, where
\begin{equation}
	\label{def:vor}
	\Vor_i(Y) = \{ x\in \Omega \mid \forall j\in \{1,\hdots,N\}, \nr{x
		- y_i} \leq \nr{x-y_j}\}.
\end{equation} 
In the second step, one moves every point $y^k_i$ towards the barycenter,
with respect to $\rho$, of the corresponding cell $V_i(Y^k)$.  This
algorithm can also be interpreted as a fixed point algorithm for
solving the first-order optimality condition for \eqref{pb:standard},
i.e. $\nabla G_N(Y) = 0$.
%%  gradient descent for the functional $G$
%% .  It relies on the following expression for the
%% minimized functional $G_N$:
%% $$ G_N(Y) = \sum_{1\leq i \leq N} \frac{1}{2}\int_{\Vor_i(Y)} \nr{x - y_i}^2
%% \dd\rho(x), $$ where $\Vor_i(Y)$ is the $i$th Voronoi cell associated
%% to the point cloud:
%% The function $G$ is differentiable when
%% all the points $y_1,\hdots,y_N$ are distinct, and its gradient with
%% respect to the point $y_i$ is given by
%% \begin{equation}\label{eq:gradG}
%%   \nabla_{y_i} G(Y) = m_i(Y) (y_i - c_i(Y)),
%% \end{equation}
%% where $m_i(Y)$ and $c_i(Y)$ are respectively the mass and barycenter
%% with respect to the probability density $\rho$ of the $i$th Voronoi cell:
%% $$ m_i(Y) = \rho(\Vor_i(Y)),~~ c_i(Y) = \frac{1}{m_i(Y)}
%% \int_{\Vor_i(Y)} x\dd\rho.$$ Equation~\eqref{eq:gradG} shows that to
%% decrease the quantization energy $G$, one wants to move points towards
%% the barycenter of their Voronoi cells. Lloyd's algorithm is a simple
%% fixed point algorithm relying on the above expression for
%% $\nabla_{y_i} G$. Given $Y^0 = (y_1^0,\hdots,y_N^0)$ one defines the
%% iterates
%% \begin{equation}\label{algo:lloyd}
%%   y_i^{k+1} = c_i(Y^k), ~~ Y^{k+1} = (y_0^{k+1},\hdots, y_N^{k+1}).
%% \end{equation}
One can show that the energy $(G_N(Y^k))_{k\geq 0}$ decreases in
$k$. The convergence of $Y^k$ towards a critical point of $F_N$ as
$k\to +\infty$ has been studied in \cite{du2006convergence}, but 
the energy of this limit critical point is not guaranteed
to be small.

In the case of the uniform quantization problem \eqref{pb:main}, one
can try to minimize the energy $F_N$ by gradient descent, defining the
iterates
\begin{equation}
  \label{eq:gf}
  Y^{k+1} = Y^k - \tau N\nabla F_N(Y^k),
\end{equation} where $\tau>0$ is the time
step. The factor $N$ in front of $\nabla F_N$ is set as a compensation
for the fact that we have, in general, $\nabla F_N(Y)=O(1/N)$. When
$\tau = 1$, one recovers a version of Lloyd's algorithm for the
uniform quantization problem, involving barycenters $B_N(Y)$ of Power cells,
rather than Voronoi cells, associated to $Y$. More precisely,   Proposition~\ref{prop:gradF} 
proves that  $\nabla F_N(Y)=(Y-B_N(Y))/N$ so that $Y^{k+1} = B_N(Y^k)$ when $\tau = 1$.
Quite surprisingly, we prove in Corollary~\ref{cor:onestep} that if
the points in the initial cloud $Y^0$ are not too close to each other,
then the uniform measure over the point cloud $Y^1 = Y^0 - N\nabla
F_N(Y^0)$ obtained after \emph{only one step} of Lloyd's algorithm is
close to $\rho$. This is illustrated in Figure~\ref{fig:lloyd}. We
prove in particular the following statement.

\begin{thm*}[Particular case of Corollary~\ref{cor:onestep}]
  Let $\rho$ be a probability density over a compact convex set
  $\Omega \subseteq \Rsp^d$, let $Y^0 = (y_1^0,\hdots,y_N^0) \in
  \Omega^d$ and assume that the points lie at some positive distance
  from one another: for some constant $c$,
  $$ \forall i\neq j, \nr{y_i - y_j} \geq c N^{-1/d},$$ corresponding
  for instance to a point cloud sampled on a regular grid. Then, the
  point cloud $Y^1 = Y^0 - N \nabla F_N(Y^0)$ obtained after one step of
  Lloyd's algorithm satisfies
  $$ \Wass_2^2(\delta_{Y^1},\rho) \leq C_{c,d,\Omega} N^{-1/d}, $$
  where $C_{c,d,\Omega}$ is a constant depending on $c,d$ and $\diam(\Omega)$.
\end{thm*}

%% To see this, we need to
%% construct the analogue of the Voronoi tesselation. Let $Y$ be a point
%% cloud with distinct points, and let $T: \Omega\to Y$ be the solution
%% to the optimal transport problem between $\rho$ and $\delta_Y$. We
%% denote $P_i(Y)$ the closure of $T^{-1}(y_i) \subseteq \Omega$, and
%% call it the $i$th \emph{power cell}. The cells $P_i(Y)$ partition
%% $\Omega$ up to a negligible set and Proposition~\ref{prop:gradF} shows
%% that the gradient of the uniform quantization energy $F$ with respect
%% to the point $y_i$ is given by
%% $$ \nabla F_N(Y) = Y - B_N(Y), $$ where $B_N(Y) = (b_i(Y))_{1\leq
%%   i\leq N}$ and where $b_i$ is the barycenter with respect to $\rho$ of the $i$th Power cell:
%% $$ b_i(Y) = N \int_{P_i(Y)} x\dd\rho.$$ This leads to a variant of
%% Lloyd's algorithm for approximately solving the uniform quantization
%% problem: starting from $Y^0 \in (\Rsp^{d})^N$ we define
%% iteratively
%% \begin{equation}\label{algo:lloyd2}
%%   Y^{k+1} = B_N(Y^k)
%% \end{equation}
%% As above, one can check that the energy $F_N$ decreases along these
%% iterations and that $Y^k$ converges as $k\to +\infty$ to a critical
%% point of $F$. Our second result analyses the behaviour of a single
%% step of Lloyd's algorithm \eqref{algo:lloyd2} for the minimization of
%% the uniform quantization energy $F$. 

%% \textcolor{red}{mettre une image}

\paragraph{Outline}
In Section~\ref{sec:deterministic}, we start by a short review of
background material on optimal transport and optimal uniform
quantization. We then establish our main result
(\cref{thm:approx_bary}) on the approximation of a measure $\rho$ by
barycenters of Power cells. This theorem yields error estimates for
one step of Lloyd's algorithm in deterministic and probabilistic
settings (Corollaries \ref{cor:onestep} and \ref{coro:proba}). In
Section~\ref{sec:gf}, we establish a Polyak-\L ojasiewicz-type
inequality (Corollary~\ref{cor:PL}) for the function $F_N =
\frac{1}{2}\Wass_2^2(\rho,\delta_Y)$ introduced in \eqref{eq:F}, and
we study the convergence of a  gradient descent algorithm
for $F_N$ (Theorem~\ref{thm:gf}). Finally, in Section~\ref{sec:num},
we report numerical results on optimal uniform quantization in
dimension $d=2$.

%% \paragraph*{Contributions}
%% Our aim is to analyze the convergence of the iterates of Lloyd's
%% algorithm \eqref{algo:lloyd2} and of the solution to the gradient flow
%% \eqref{eq:gfintro} to $\rho$ in the limit where the number of atoms
%% $N$ tends to $+\infty$. In our estimates, the notation $A \lesssim
%% B$ means that $A$ is less than a constant depending only on the
%% dimension $d$ or the domain $\Omega$ times $B$.

%% Our second result (Corollary~\ref{coro:gf}) concerns the solution to
%% the gradient flow \eqref{eq:gfintro}.  Assume that $\nr{y_i^0 -
%%   y_j^0}\geq N^{-1/d}$, corresponding to a point cloud sampled on a
%% regular grid, and consider the solution $Y(t)$ of the gradient flow
%% \eqref{eq:gfintro}. Then, we prove that if one chooses a time depending logarithmically on $N$,
%% $$ t_{N} \simeq A + B\log(N),$$
%% then we obtain an upper bound of the form
%% $$ \Wass_2^2(\rho, \delta_{Y(t_N)}) \lesssim
%% N^{-\frac{1}{d^2}}. $$ This estimates is worse than the one for a
%% single step of Lloyd's algorithm, but already relies on a non-trivial
%% behaviour (Lemma~\ref{lem:stretch}) of the trajectories of the
%% gradient flow, controlling the minimum distance between points at
%% positive time.

%% In Section~\ref{sec:prob}, we discuss a probabilistic version of the
%% convergence result for Lloyd's algorithm for a random initial point
%% cloud. Section~\ref{sec:num} present numerical experiments both in the
%% deterministic and probabilistic setting.

%% file: sec2-deterministic.tex
\section{Lloyd's algorithm for optimal uniform quantization}
\label{sec:deterministic}
  
\paragraph*{Optimal transport and Kantorovich duality}
In this section we briefly review Kantorovich duality and its relation
to semidiscrete optimal transport. The cost is fixed to $c(x,y) =
\nr{x-y}^2$, and we assume that $\rho$ is a probability density over a
compact convex domain $\Omega$.  In this setting, Brenier's theorem
implies that given any probability measure $\mu$ supported on
$\Omega$, the optimal transport plan between $\rho$ and $\mu$,
i.e. the minimizer $\pi$ in the definition of the Wasserstein distance
\eqref{eq:wass} with $p=2$, is induced by a transport map
$T_\mu:\Omega\to\Omega$, meaning $\pi=(T_\mu,Id)_\#\rho$.  One can derive
an alternative expression for the Wasserstein distance using
Kantorovich duality, which leads to a more precise description of
the optimal transport map
\cite[Theorem~1.39]{santambrogio2015optimal}:
\begin{equation}\label{eq:kantdual}
  \Wass_2^2(\rho,\mu)=\max_{\phi: Y\to\Rsp} \int_{\R^d} \phi \dd\mu
  + \int_\Omega\phi^c \dd\rho,
\end{equation}
where $\phi^c(x) = \min_{i} c(x,y_i) - \phi_i$.  When $\mu = \delta_Y$
is the uniform probability measure over a point cloud $Y =
(y_1,\hdots,y_N)$ containing $N$ distinct points, we set $\phi_i =
\phi(y_i)$ and we define the \emph{$i$th Power cell} associated to the
couple $(Y,\phi)$ as
\begin{equation*}
	\Pow{i}{Y}{\phi}=\{x\in\Rsp^d \mid \forall j\in\{1,\hdots,N\},~\norm{x-y_i}^2-\phi_i
	\leq \norm{x-y_j}^2-\phi_j\}.
\end{equation*}
Then, the Kantorovich dual \eqref{eq:kantdual} of the optimal transport
problem between $\rho$ and $\delta_Y$ turns into a finite-dimensional
concave maximization problem
\begin{equation} \label{eq:dual}
	\Wass_2^2(\mu,\rho) 
	=  \max_{\phi\in\R^N} \sum_{i=1}^N\frac{\phi_i}{N} +\int_{\Pow{i}{Y}{\phi}}\left(\norm{x-y_i}^2-\phi_i \right)\dd\rho(x)
\end{equation}
By Corollary~1.2 in \cite{KitMerThi17}, a vector $\phi\in\R^N$ is
optimal for this maximization problem if and only if the potential $\phi$ is such
that each Power cell contain the same amount of mass, i.e. if
\begin{equation} \label{eq:otdual}
	\forall i\in\{1,\hdots,N\},~~\rho(\Pow{i}{Y}{\phi})=\frac{1}{N},
\end{equation}
From now on, we denote $P_i(Y) = \Pow{i}{Y}{\phi} \cap \Omega$, where
$\phi \in \Rsp^N$ satisfies \eqref{eq:otdual}. The optimal transport
map $T_Y$ between $\rho$ and $\delta_Y$ sends every Power cell
$P_i(Y)$ to the point $y_i$, i.e. it is  defined $\rho$-almost
everywhere by $\left.{T_Y}\right|_{P_i(Y)} = y_i$.
We refer again to the introduction of \cite{KitMerThi17} for more
details.

\paragraph*{Optimal uniform quantization}
In this article, we study the behaviour of the squared Wasserstein
distance between the (fixed) probability density $\rho$ and a uniform
finitely supported measure $\delta_{Y}$ where $Y = (y_1,\hdots,y_N)$
is a cloud of $N$ points, in terms of variations of $Y$. As in equation
\eqref{eq:F}, we denote $F_N = \frac{1}{2}
\Wass_2^2(\rho,\cdot)$. Proposition 21 in \cite{merigot2016minimal}
  gives an expression for the gradient of $F$, and proves its
  semiconcavity. We recall that $F$ is called $\alpha$--semiconcave,
  with $\alpha\geq 0$, if the function $F -
  \frac{\alpha}{2}\nr{\cdot}^2$ is concave. We denote $\Diag_N$ the
  generalized diagonal
$$\Diag_N=\{Y\in(\R^d)^N \mid \exists i\neq j \hbox{ s.t. } y_i = y_j\}.$$

%  \textcolor{red}{$F_N$ est $\frac1N$-semiconcave, non (Quentin) ? Oui (Clément)}
\begin{prop}[Gradient of $F_N$] \label{prop:gradF}
	The function $F_N$ is $\frac{1}{N}$--semiconcave on $(\R^d)^N$ and is of
        class $\mathcal{C}^1$ on $(\Rsp^d)^{N}\setminus \Diag_N$. In
        addition, for any $Y\in\Diag_N$ one has
	\begin{equation} \label{eq:gradF}
          \forall Y\in (\Rsp^d)^{N}\setminus
        \Diag_N, 
	\nabla F_N(Y) = \frac1N( Y - B_N(Y)), \hbox{ where } B_N(Y) = (b_1(Y),\hdots,b_N(Y))
        \end{equation}
        and where $b_i(Y)$ is the barycenter of the $i$th power cell,
        i.e. $b_i(Y) = N\int_{P_i(Y)}\dd\rho(x)$.
\end{prop}
%% \textcolor{red}{Pourquoi indiquer la dépendance de B en N? Techniquement, l'information est contenue dans Y et ça donne l'impression que la définition change beaucoup d'un N à l'autre?}

It is not difficult to prove that $F_N$ admits at least one minimizer,
and that this minimizer $Y$ satisfies the first-order optimality
condition $Y = B_N(Y)$. A point cloud that satisfies this condition is
called \emph{critical}.

\begin{rmk}[Upper bound on the minimum of $F_N$]\label{rmk:rateminF}
    We note from \cite[Proposition~12]{merigot2016minimal}  that when $\rho$ is supported on a compact subset of $\Rsp^d$, then
    \begin{equation}\label{eq:rateminF}
       \min F_N = \min_{Y \in (\Rsp^{d})^N} \frac{1}{2}\Wass_2^2(\rho,\delta_Y)    \lesssim
    \begin{cases}
      N^{-\frac{2}{d}}&\hbox{ if } d>2\\
        N^{-1} \log N &\hbox{ if } d=2\\
          N^{-1}&\hbox{ if } d=1.\\
    \end{cases}
    \end{equation}
  These upper bounds may not be tight, in particular when $\rho$ is
  separable (see \cref{app:gaussian_behaviour}).
\end{rmk}
  
\begin{minipage}{.7\textwidth}
\begin{rmk}[High energy critical points] \label{rmk:badcritical}
  On the other hand, since $F_N$ is not convex, this first-order
  condition is not sufficient to have a minimizer of $F_N$.  For
  instance, if $\rho \equiv 1$ on the unit square $\Omega = [0,1]^2$,
  one can check that the point cloud
	$$Y_N = \left( \left(\frac{1}{2N}, \frac{1}{2}\right),
  \left(\frac{3}{2N}, \frac{1}{2}\right),\hdots,
  \left(\frac{2N-1}{2N},\frac{1}{2}\right) \right)$$
\end{rmk}
\end{minipage}\hfill
\begin{minipage}{.23\textwidth}
%\begin{wrapfigure}{r}{0.21\textwidth}
  \includegraphics[width=\textwidth]{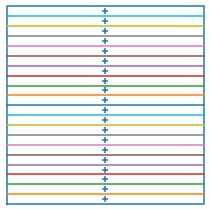}
  %\end{wrapfigure}
\end{minipage}\hfill

\noindent is a critical
  point of $F_N$ but not a minimizer of $F_N$. In fact, this critical
  point becomes arbitrarily bad as $N\to +\infty$ in the sense that
        $$ \lim_{N\to +\infty} \frac{F_N(Y_N)}{\min F_N} = +\infty.$$
  On the other hand, we note that the point cloud $Y_N$ is highly
  concentrated, in the sense that the distance between consecutive
  points in $Y_N$ is $\frac{1}{2N}$, whereas in an evenly distributed
  point cloud, one would expect the minimum distance between points to
  be of order $N^{-1/d}$.
  
\paragraph{Gradient descent and Lloyd's algorithm}
One can find a critical point of $F_N$ by following the discrete
gradient flow of $F_N$, defined in \eqref{eq:gf}, starting from an
initial position $Y^0 \in (\Rsp^{d})^N\setminus \Diag_N$. Thanks to
the expression of $\nabla F_N$ given in Proposition~\ref{prop:gradF},
the discrete gradient flow may be written as
\begin{equation}
	\label{eq:discr_flow}
		Y^{k+1} = Y^k + \tau_N (B_N(Y^k) - Y^k),
\end{equation}
where $\tau_N$ is a fixed time step.  For $\tau_N = 1$, one recovers a
variant of Lloyd's algorithm, where one moves every point to the
barycenter of its Power cell $P_i(Y^k)$ at each iteration:
\begin{equation}\label{algo:lloyd2}
  Y^{k+1} = B_N(Y^k)
\end{equation}
We can state the following result about Lloyd's algorithm for the uniform
quantization problem, whose proof is postponed to the appendix.

\begin{prop}\label{prop:cvlloyd}
  Let $N$ be a fixed integer and $(Y^k)_{k\geq 0}$ be the iterates of \eqref{algo:lloyd2}, with
  $Y^0 \not\in\Diag_N$. Then, the energy $k\mapsto F_N(Y^k)$ is
  decreasing, and $\lim_{k\to +\infty} \nr{\nabla F_N(Y^k)} = 0$. Moreover, the sequence $(Y^k)_{k\geq 0}$ belongs to a compact
  subset of $(\Rsp^{d})^N\setminus \Diag_N$ and every limit point of a converging subsequence of it is a critical point for $F_N$.
\end{prop}

Experiments suggest that following the discrete gradient flow of
$F_N$ does not bring us to high energy critical points of $F_N$, such
as those described in Remark~\ref{rmk:badcritical}, unless we started
from an adversely chosen point cloud. The following theorem and its
corollaries, the main results of this article, backs up this
experimental evidence. It shows that if the point cloud $Y$ is not too
concentrated, then the uniform measure over the barycenters of the
power cells, $\delta_{B_N(Y)}$, is a good quantization of the
probability density $\rho$, i.e. it bounds the quantization error
after one step of Lloyd's algorithm \eqref{algo:lloyd2}.

We will use the following notation for $\eps>0$:
$$ I_\eps(Y) = \{ i\in\{1,\hdots,N\}\mid \forall j \neq i, \nr{y_i - y_j}\geq \eps\}. $$
$$ \Diag_{N,\eps} = \{ Y \in (\Rsp^{N})^d \mid \exists i \neq j,
\nr{y_i - y_j}\leq \eps \}. $$ Note that $\Diag_{N,\eps}$ is an
$\eps$-neighborhood around the generalized diagonal $\Diag_N$.

\begin{thm}[Quantization by barycenters]
  \label{thm:approx_bary}
  Let $\Omega\subseteq \Rsp^d$ be a compact convex set, $\rho$ a
  probability density on $\Omega$ and consider a point cloud
  $Y=(y_1,\dots,y_N)$ in $\Omega^N \setminus\Diag_N$.  Then, for all $0 <\eps\leq 1$,
	\begin{equation}
		\Wass^2_2\left(\rho,\delta_{B_N(Y)}\right)\leq C_{d,\Omega}\left(\frac{\eps^{1-d}}{N} + 1-\frac{\Card(I_\eps(Y))}{N}\right).		
	\end{equation}
where
$C_{d,\Omega}= \frac{2^{2d-1}}{\omega_{d-1}} (\diam(\Omega)+1)^{d+1}$
and where $\omega_{d-1}$ is the volume of the unit ball in $\Rsp^{d-1}$.
\end{thm}

The proof relies on arguments from convex geometry. In
particular, we denote $A\oplus B$ the Minkowski sum of sets: $A \oplus
B = \{ a+b \mid (a,b)\in A\times B\}$.
\begin{proof}
	Let $\phi^1 \in\Rsp^N$ be the solution to the dual Kantorovich
        problem \eqref{eq:otdual} between $\rho$ and $\delta_Y$. We
        let $\phi^t = t \phi^1$ and we denote
        $P_i^t=\Pow{i}{Y}{\phi^t}\cap \Omega'$ the $i$th Power cell
        intersected with the slightly enlarged convex set $\Omega' =
        \Omega\oplus \B(0,1)$.  This way, $P_i^1 \supseteq P_i(Y)$
        whereas $P_i^0$ is in fact the intersection of the $i$-th
        Voronoi cell defined in \eqref{def:vor} with $\Omega'$.
	
	We will now prove an upper bound on the sum of the diameters
        of the cells $P_i(Y)$ whose index lies in $I_\eps(Y)$. First,
        we notice the following inclusion, which holds for any $t\in
        [0,1]$:
	\begin{equation}\label{eq:minksum}
		(1-t) P_i^0 \oplus  tP_i^1 \subseteq P_i^t,
	\end{equation}
	Indeed, let $x^0 \in P_i^0$ and $x^1\in P_i^1$, so that for all $j\in
	\{1,\hdots,N\}$ and $k \in \{0,1\}$,
	$$\nr{x^k - y_i}^2 - \phi_i \leq \nr{x^k - y_j}^2 - \phi_j.$$
	Expanding the squares and substracting $\nr{x^k}^2$ on both sides
	these inequalities become \emph{linear} in $\phi_i,\phi_j$ and $x^k$, implying %% :
	directly that $x^t = (1-t)x^0 + tx^1\subseteq P_i^t$ as desired.
	
 For any index $i\in I_\eps$, the point $y_i$ is at distance at least
 $\eps$ from other points, implying that $\B(0,\frac{\eps}{2})$ is
 contained in the Voronoi cell $V_i(Y)$ with $\Omega'$. Using that
 $P_i^0 = V_i(Y)\cap \Omega'$, that $\Omega' = \Omega \oplus \B(0,1)$
 and that $y_i \in \Omega$, we deduce that $P_i^0$ contains the same
 ball. On the other hand, $P_i^1$ contains a segment $S_i$ of length
 $\diam(P_i^1)$ and inclusion \eqref{eq:minksum} with $t=\frac12$
 gives $$\frac{1}{2}(\B(y_i,\epsilon/2)\oplus S_i)\subseteq
 P_i^{1/2}.$$ The Minkowski sum in the left-hand side contains in
 particular the product of a $(d-1)$-dimensional ball of radius
 $\eps/2$ with an orthogonal segment with length $\diam(P_i^1)\geq
 \diam(P_i(Y))$. Thus,
 $$\frac{1}{2^d}\left(\omega_{d-1} \frac{\eps^{d-1}}{2^{d-1}}
 \diam(P_i(Y))\right)\leq |P_i^\frac{1}{2}|.$$ Using that the Power
 cells $P_i^{\frac{1}{2}}$ form a tesselation of the domain $\Omega'$, we
 therefore obtain
 \begin{equation}
   \label{eq:bound_diam} \sum_{i\in I_\eps(Y)} \diam(P_i(Y)) \leq \frac{2^{2d-1}}{\omega_{d-1}}  \abs{\Omega'} \eps^{1-d} \leq \frac{2^{2d-1}}{\omega_{d-1}} (\diam(\Omega)+1)^d \eps^{1-d}
 \end{equation}
 We now estimate the transport cost between $\delta_{B}$ and the
 density $\rho$, where $B=B_N(Y)$. The transport cost due to the
 points whose indices do not belong to $I_\eps(Y)$ can be bounded in a
 crude way by
 $$ \sum_{i\not\in I_\eps(Y)} \int_{P_i(Y)} \nr{x - y_i}^2 \dd\rho(x)
 \leq (1-\frac{\Card I_\eps(Y)}{N}) \diam(\Omega)^2. $$ Note that we
 used $\rho(P_i(Y)) = \frac{1}{N}$.  On the other hand, the transport
 cost associated with indices in $I_\eps(Y)$ can be bounded using
 \eqref{eq:bound_diam} and $\diam(P_i(Y))\leq \diam(\Omega)$:
$$
\begin{aligned}
   \sum_{i \in I_\eps(Y)} \int_{P_i(Y)} \nr{x - y_i}^2 \dd\rho(x)
&\leq\frac{1}{N} \sum_{i\in I_\eps(Y)} \diam(P_i(Y))^2\\
   &\leq \frac{1}{N} \diam(\Omega)\sum_{i\in I_\eps}\diam(P_i(Y)) \\
   &\leq \frac{2^{2d-1}}{\omega_{d-1}} (\diam(\Omega)+1)^{d+1} \frac{\eps^{1-d}}{N}
\end{aligned}
$$
	%% $$ \nr{y_i}^2 - 2\sca{x^k}{y_i}  - \phi_i \leq \nr{y_j}^2 - 2\sca{x^k}{y_j} - \phi_j.$$
	%% We deduce that
	%% Take $\phi^1$
	%% a vector of Kantorovich potentials for the transport from $\mu_N$ to
	%% $\rho^0$, and for every $i$, the ith optimal Power cell
	%% $P_i^1=Pow_i(Y,\phi^1)$ as well as the ith Voronoi cell
	%% $P_i^0=Pow_i(Y,0_{\Rsp^N})$ in the tesselation associated with $y$
	%% and $\phi^0=0_{\R^N}$. Then for the intermediate
	%% $\phi^{1/2}=\frac{1}{2}\phi^0+\frac{1}{2}\phi^1$, 
	%%       Fix $\alpha<1$ and consider a Pow cell $Pow_i(Y,\phi^1)$ . Denoting by $N_{ab}$ the number of `abnormal" cells in which there is such a segment, one has, $N_{ab}\epsilon^{d-1+\alpha}\lesssim1$ with a constant depending only on the dimension and $\Omega$.\newline	
	%	Now, we can conclude:
	In conclusion, we obtain the desired estimate:
	\begin{equation*}
			\Wass_2^2\left(\rho,\delta_{B_N(Y)}\right) 
			\leq
\frac{2^{2d-1}}{\omega_{d-1}} (\diam(\Omega)+1)^{d+1} \frac{\eps^{1-d}}{N} + \diam(\Omega)^2\left(1-\frac{\Card I_\eps}{N}\right).\qedhere
			%\\
			%			& \lesssim \frac{\epsilon^{\frac{2(1-d)}{3}}}{N^{\frac{2}{3}}}
	\end{equation*}
	%	taking $\alpha=\frac{1-d}{3}$.
\end{proof}

This theorem could be extended \emph{mutatis mutandis} to the case
where $\rho$ is a general probability measure (i.e. not a
density). However, this would imply some technical complications in
the definition of the barycenters $b_i$ by introducing a
disintegration of $\rho$ with respect to the transport plan $\pi$.  

\paragraph{Consequence for Lloyd's algorithm~\eqref{algo:lloyd2}}
In the next corollary, we assume that any pair of distinct points in
$Y_N \in (\Rsp^{d})^N$ is bounded from below by $\eps_N \geq
CN^{-\beta}$, implying that $I_{\eps_N}(Y_N) = N$. This corresponds to
the value one could expect for a point set uniformly sampled from a
set with Minkowski dimension $\beta$.  When $\beta > d-1$, the
corollary asserts that one step of Lloyd's algorithm is enough to
approximate $\rho$, in the sense that the uniform measure
$\delta_{B_N(Y_N)}$ over the barycenters converges towards $\rho$ as
$N\to +\infty$.

\begin{cor}[Quantization by barycenters, asymptotic case] \label{cor:onestep}
	Assume  $\epsilon_N\geq C\cdot N^{-1/\beta}$ with $C,\beta>0$. Then,
	with $\alpha=1-\frac{d-1}{\beta}$  
	\begin{equation} \label{eq:onestep}
		\forall Y\in (\Rsp^d)^N\setminus \Diag_{\eps_N},\quad \Wass^2_2(\rho, \delta_{B_N(Y)})
		\leq \frac{C_{d,\Omega}}{C^{d-1}}N^{-\alpha}, 
	\end{equation}
	and in particular, if $\beta > d-1$,
	\begin{equation}
		\label{eq:conv_beta}
		\lim_{N\to+\infty} \max_{Y\in (\Rsp^d)^N\setminus\Diag_{\eps_N}} \Wass^2_2(\rho, \delta_{B_N(Y)}) = 0. 
	\end{equation}
\end{cor}

\begin{rmk}[Optimality of the exponent when $\beta=d$] \label{rmk:optimassump}
	There is no reason to believe that the exponent in the upper
        bound \eqref{eq:onestep} is optimal in general.  However, it
        seems to be optimal in a ``worst-case sense'' when $\beta=d$.
        More precisely, we show the following result in
        \cref{app:gaussian_behaviour}: for any $\delta\in(0,1)$, and
        for every $N=n^d$ ($n\in\N$) there exists a sequence of
        separable probability densities $\rho_N$ over $X = [-1,1]^d$
        ($\rho_N$ is a truncated Gaussian distributions, whose
        variance converges to zero slowly as $N\to +\infty$) such that
        if $Y_N$ is a uniform grid of size $n\times \cdot \times n =
        N^d$ in $X$, then
        $$\Wass_2^2(\delta_{B_N(Y_N)},\rho_N) \geq CN^{-\frac
          {(2-\delta)}{d}},$$
        where $C$ is independent of $N$. On the other hand, in this setting every
        point in $Y_N$ is at distance at least $C N^{-1/d}$ from any
        other point in $Y_N$. Applying \cref{cor:onestep} with $\beta=d$, i.e. $\alpha = \frac{1}{d}$, we get
        $$\Wass_2^2(\delta_{B_N(Y_N)},\rho_N) \leq C' N^{-\frac 1d}.$$
        Comparing this upper bound on
        $\Wass_2^2(\delta_{B_N(Y_N)},\rho_N)$ with the above lower
        bound, one sees that is is not possible to improve the
        exponent.
\end{rmk}
\begin{rmk}[Optimality of \eqref{eq:conv_beta}]\label{rmk:optimbeta}
  The assumption $\beta>d-1$ for \eqref{eq:conv_beta}
is tight: if $\rho$ is the Lebesgue measure on $[0,1]^d$, it is
possible for to construct a point cloud $Y_N$ with $N$ points on the
$(d-1)$-cube $\{\frac{1}{2}\}\times [0,1]^{d-1}$ such that
distinct point in $Y_N$ are at distance at least $\epsilon_N \geq
C\cdot N^{-1/(d-1)}$. Then, the barycenters $B_N(Y_N)$ are also
contained in the cube, so that $\Wass_2^2(\rho,\delta_{B_N(Y_N)}) \geq \frac{1}{12}$.
%In
%particular \eqref{eq:conv_beta} does not hold and moreover $\lim_{N\to
%  +\infty} \frac{F_N(B_N(Y_N))}{\min F_N} = +\infty.$
\end{rmk}

The next corollary is a probabilistic analogue of
Corollary~\ref{cor:onestep}, assuming that the initial point cloud $Y$
is drawn from a probability density $\sigma$ on $\Omega$. Note that
$\sigma$ can be distinct from $\rho$. The proof of this corollary
relies on McDiarmid's inequality to quantify the  proportion of
$\eps$-isolated points in a point  cloud that is  drawn randomly and independently
from $\sigma$. The proof of this result is in
Appendix~\ref{app:proba}.

\begin{cor}[Quantization by barycenters, probabilistic case]
  \label{coro:proba}
	Let $\sigma\in\LL^\infty(\Omega)$ and let $X_1,...,X_N$ be
	i.i.d. random variables with distribution
	$\sigma\in\LL^\infty(\R^d)$. Then, there exists a constant
	$C>0$ depending only on $\|\sigma\|_{\LL^\infty}$ and $d$,
	such that for $N$ large enough,	
	$$\Prob\left(W_2^2\left(\bary{b_i^{X}},\rho\right)\lesssim N^{-\frac{1}{2d-1}} \right)\geq 1-e^{-CN^{\frac{2d-3}{2d-1}}}$$
\end{cor}

%% file: sec3-gradientflow.tex
\begin{figure*}[t]                                                                                      
  \begin{minipage}{.48\textwidth}
    \centering        
    \vskip0cm
        \includegraphics[width=0.22\textwidth]{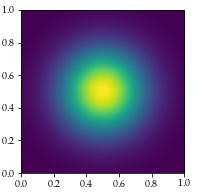}
        \includegraphics[width=0.18\textwidth]{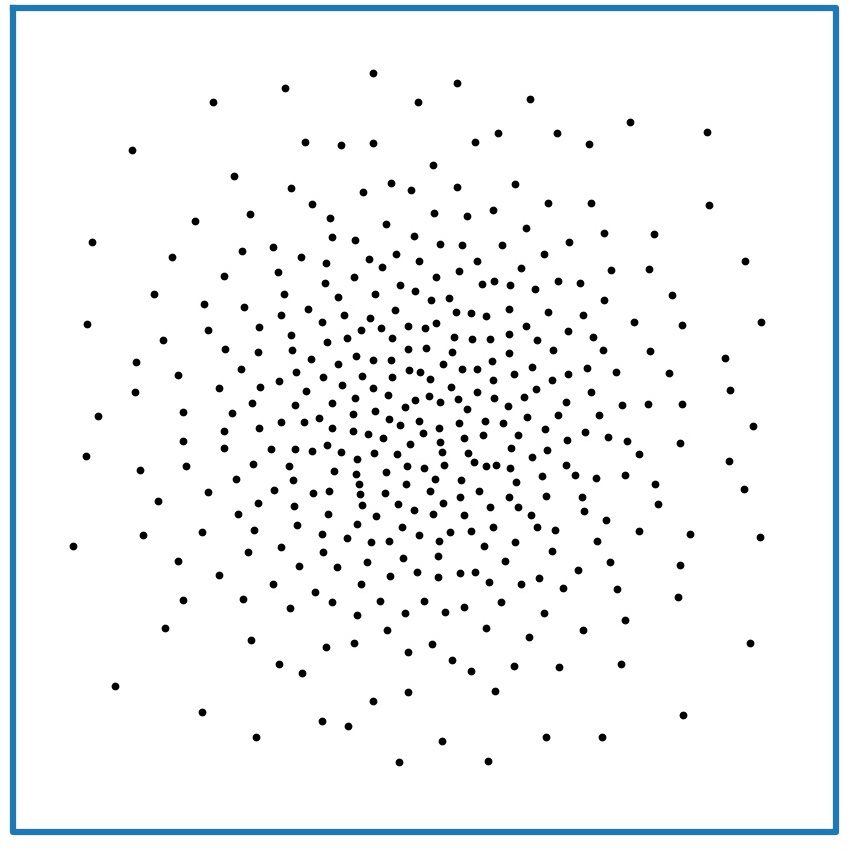}
        \includegraphics[width=0.18\textwidth]{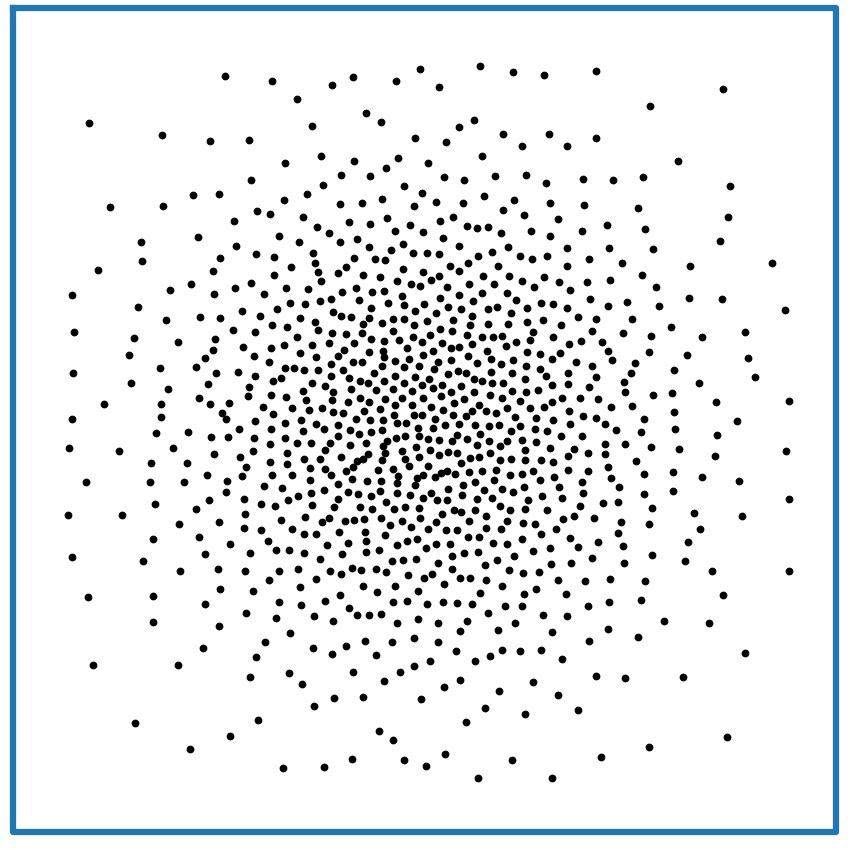}
        \includegraphics[width=0.18\textwidth]{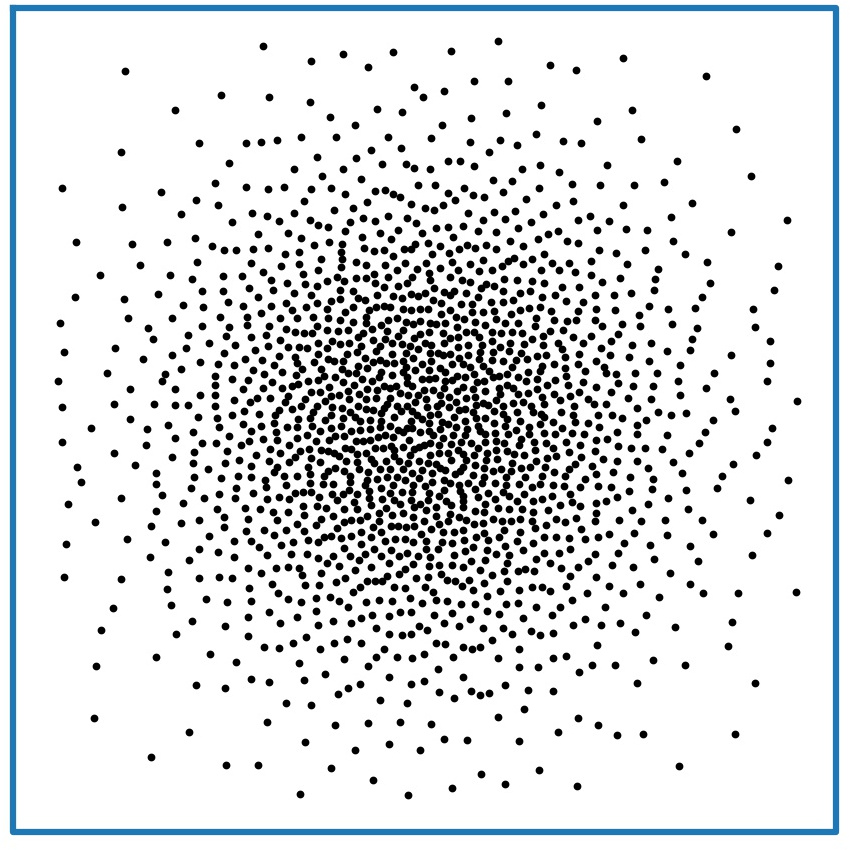}
        \includegraphics[width=0.18\textwidth]{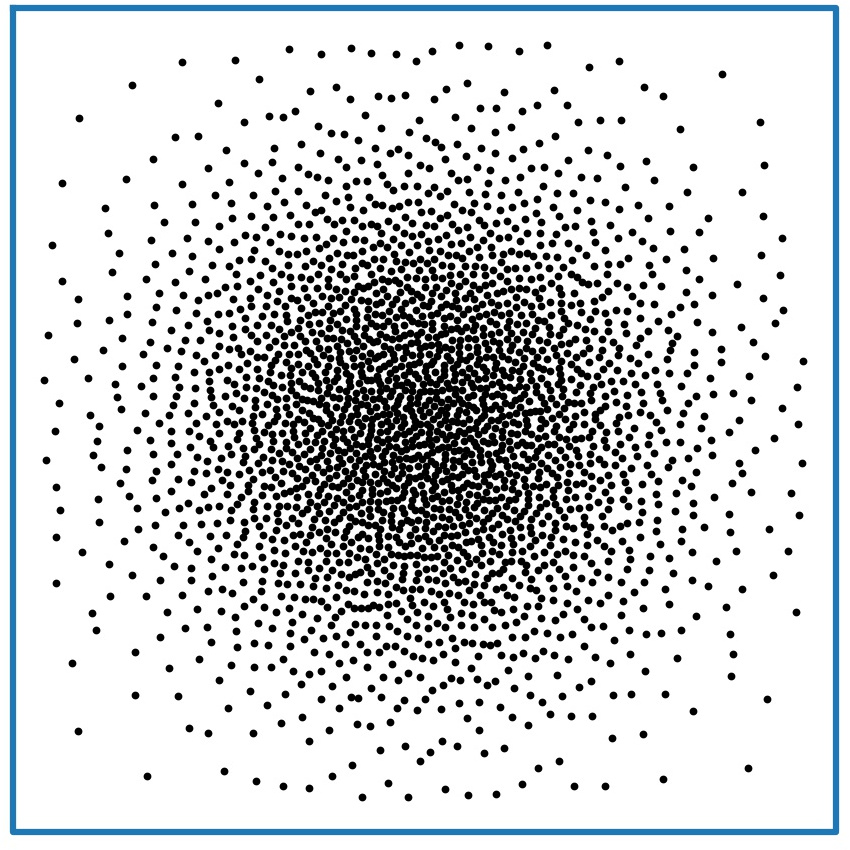}\\
        \includegraphics[width=0.6\textwidth]{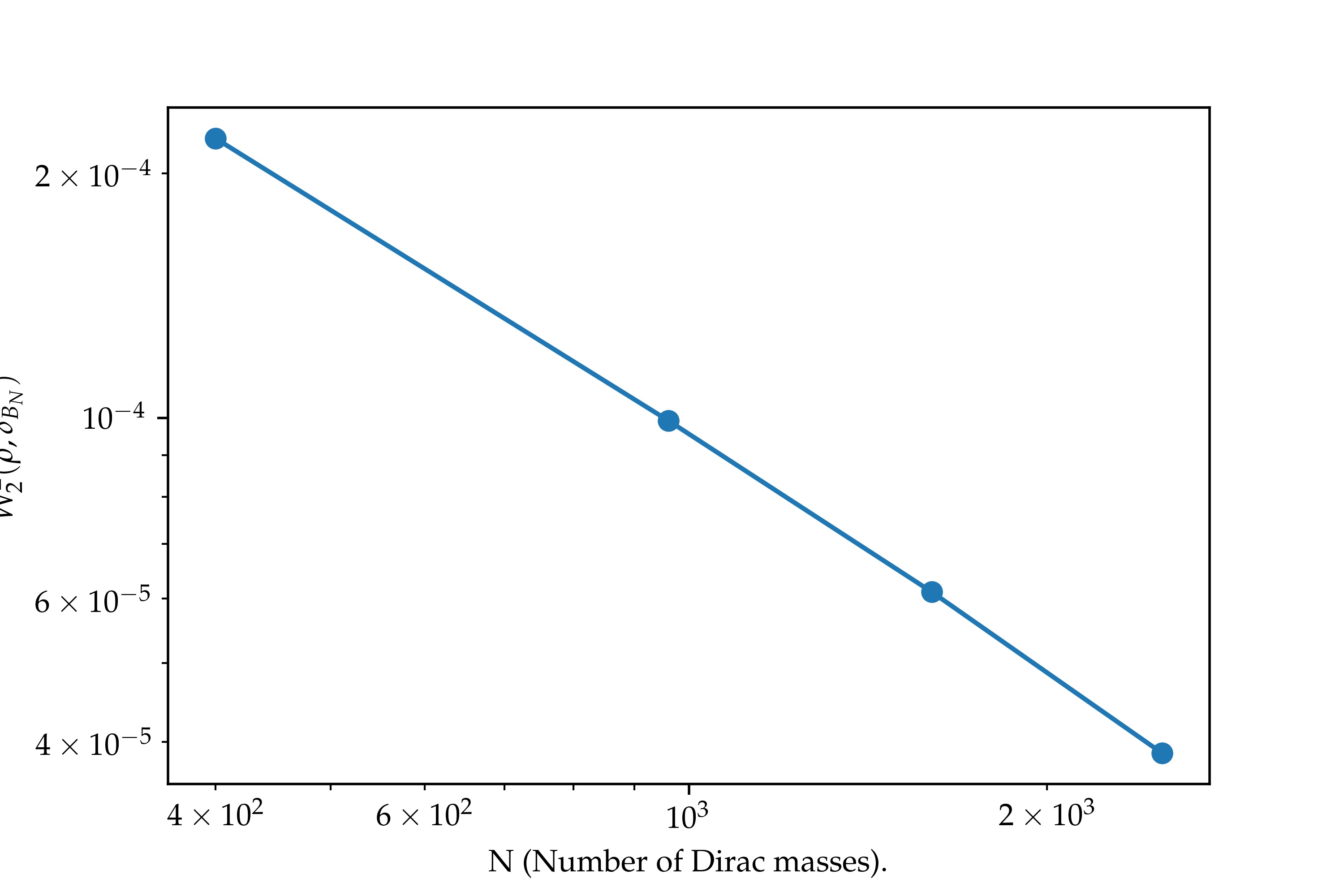}
  \end{minipage}                                                                                        \hfill\vline\hfill
	\begin{minipage}{.48\textwidth}   
		\vskip0pt                                          \centering
	    \includegraphics[width=0.22\textwidth]{figures/Gaussian_map_test.jpeg}                                           
	    \includegraphics[width=0.18\textwidth]{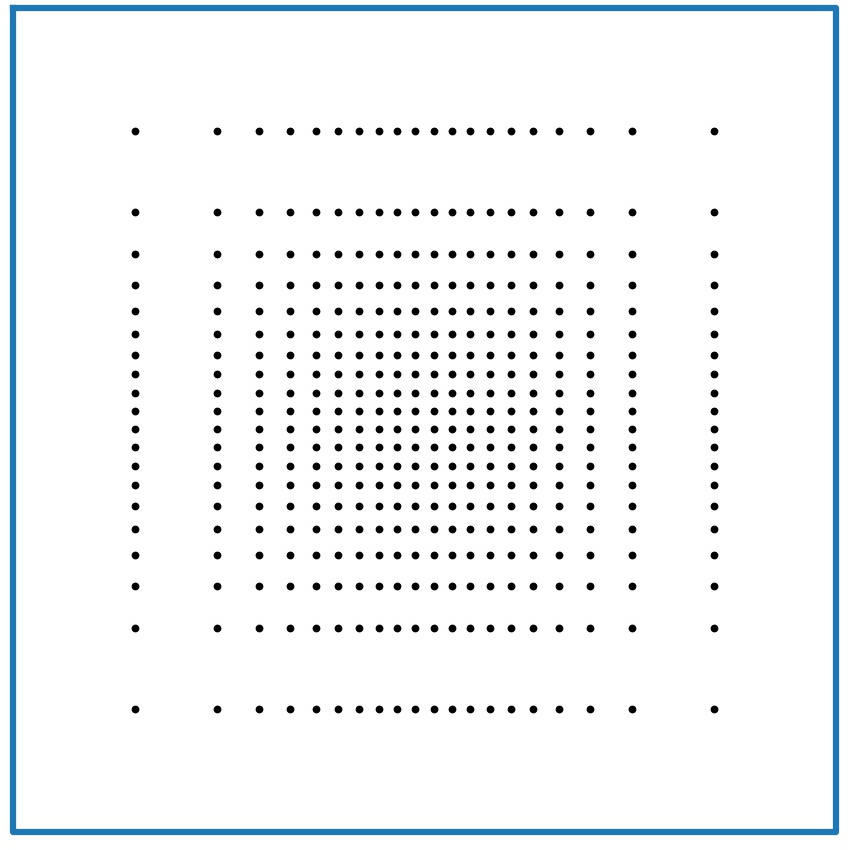}                             
	    \includegraphics[width=0.18\textwidth]{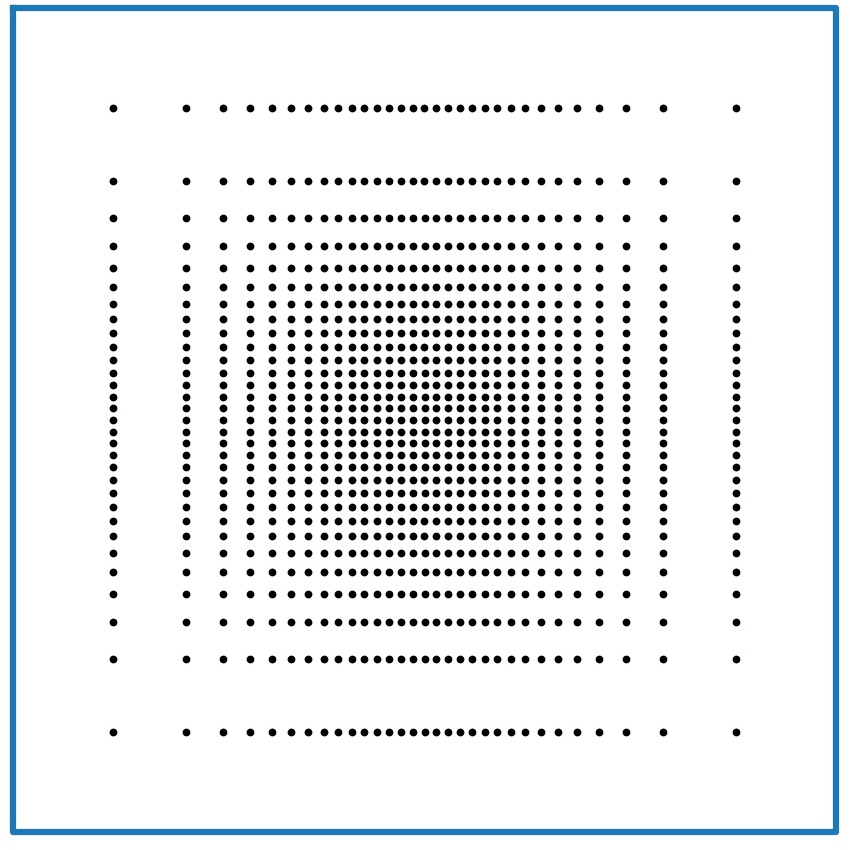}
            \includegraphics[width=0.18\textwidth]{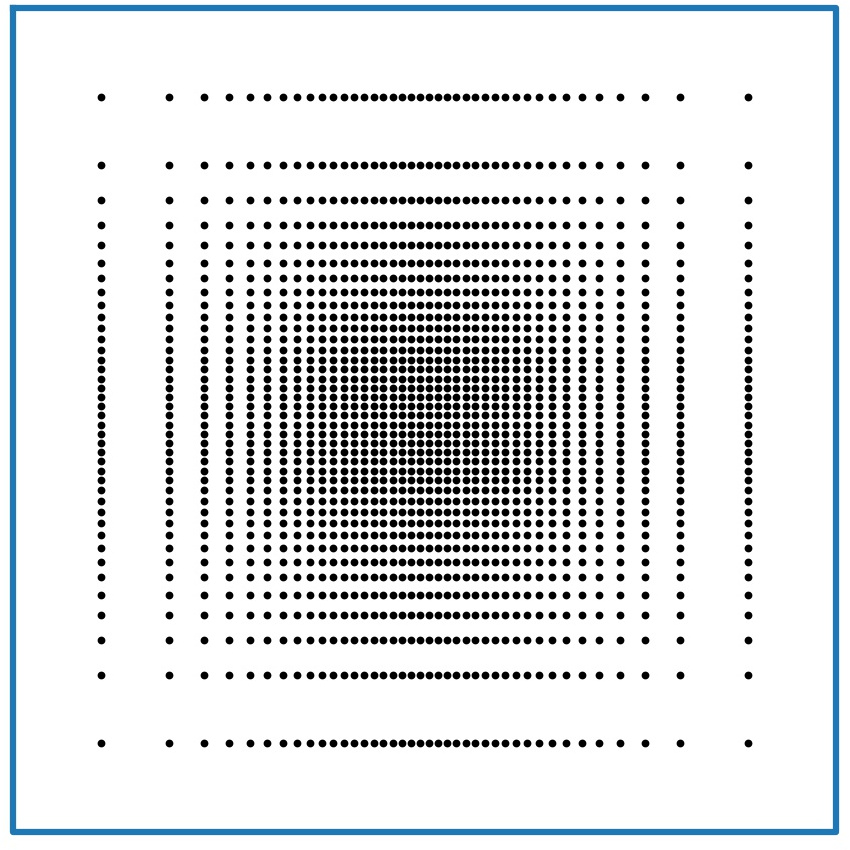}
	    \includegraphics[width=0.18\textwidth]{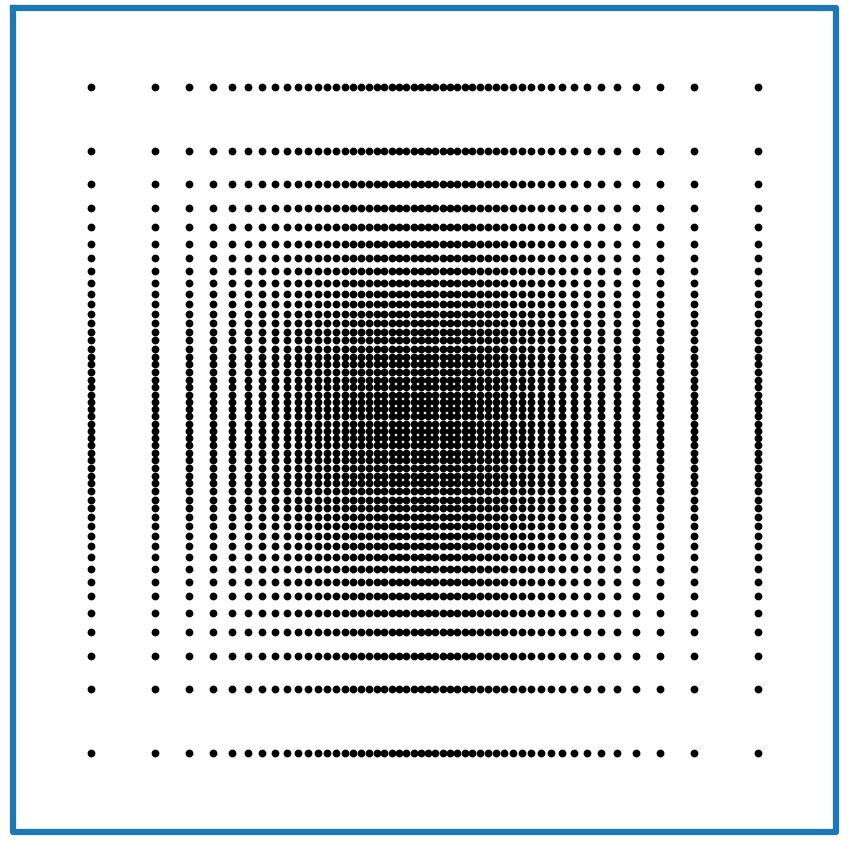}                              
	    \\                             
	    \includegraphics[width=0.6\textwidth]{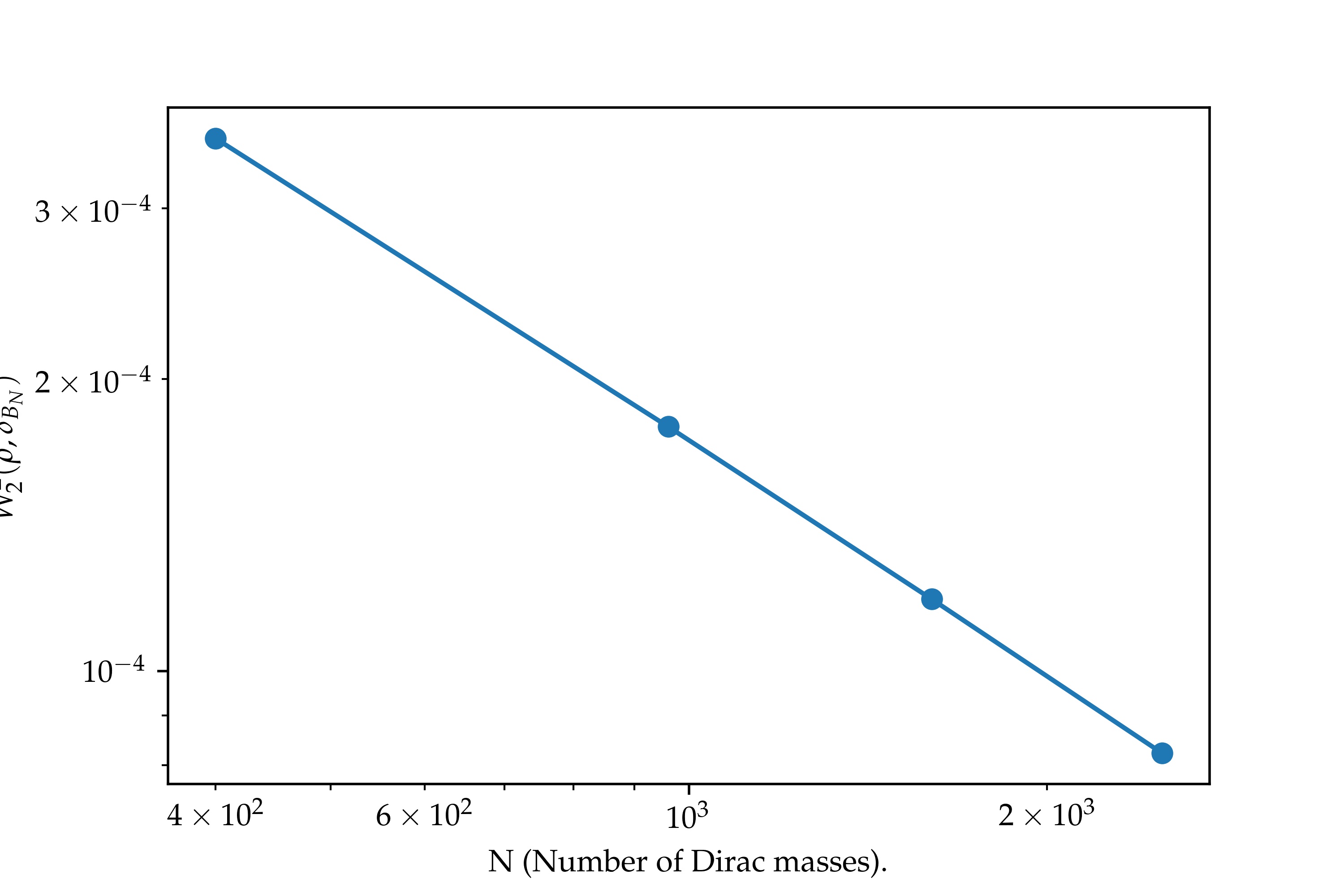}
  \end{minipage}                                                                                       
  \caption{\label{fig:gaussian}Optimal quantization of a Gaussian
    truncated to the unit square. On the left, the initial point cloud
    $Y_N$ is drawn randomly and \emph{uniformly} from $[0,1]^2$, while
    on the right $Y_N$ is on a regular grid. The top row displays the
    point clouds obtained after one step of Lloyd's algorithm. The
    bottom row displays the quantization error after one step of
    Lloyd's algorithm $F_N(B_N(Y_N))$ as a fuction of the number of
    points. We get $F_N(B_N(Y_N))\simeq N^{-0.95}$ when $Y_N$ is a random
    uniform point cloud in $[0,1]^N$ and $F_N(B_N(Y_N)\simeq N^{-0.8}$
    when $Y_N$ is a regular grid.}
%  \vspace{-.5cm}
\end{figure*}

\section{Gradient flow and a Polyak-\L ojasiewicz-type inequality}\label{sec:gf}

%\paragraph{Interpretation of Theorem~\ref{thm:approx_bary} as a Polyak-\L ojasiewicz inequality}
Theorem~\ref{thm:approx_bary} can be interpred as a modified 
Polyak-\L ojasiewicz-type (P\L{} for short) inequality for the function $F_N$. The usual
P\L{} inequality for a differentiable function $F: \Rsp^D\to \Rsp$ is
of the form
$$ \forall Y\in\Rsp^D,\quad F(Y) - \min F \leq C \nr{\nabla
  F(Y)}^2, $$ where $C$ is a positive constant. This inequality has
been originally used by Polyak to prove convergence of gradient
descent towards the global minimum of $F$. Note in particular that
such an inequality implies that any critical point of $F$ is a global
minimum of $F$. By Remark~\ref{rmk:badcritical}, $F_N$ has critical
points that are not minimizers, so that we cannot expect the standard
P\L{} inequality to hold. What we get is a similar inequality relating
$F_N(Y)$ and $\nr{\nabla F_N(Y)}^2$ but with a  term involving
the minimimum distance between the points in place of $\min F_N$.
\begin{cor}[Polyak-\L ojasiewicz-type inequality]\label{cor:PL}
  Let $Y\in(\Rsp^d)^N \setminus \Diag_{N,\eps}$. Then, 
  \begin{equation} \label{eq:PL}
    F_N(Y) - C_{d,\Omega}\frac 1N \left(\frac{1}{\eps}\right)^{d-1} \leq  N\nr{\nabla F_N(Y)}^2
  \end{equation}
  
%% 	\begin{split}
%% 	  \norm{\nabla (N.F)(Y)}^2=&\norm{B_N(Y)-Y}^2\\
%% 											\geq& N(W_2^2(\delta_Y,\rho)-W_2^2(\delta_{B_N(Y)},\rho))\\
%% 											\geq& N.F(Y)-C_{d,\Omega}\epsilon^{1-d}
%% 	\end{split}
%% \end{equation*}
\end{cor}
We note that when $\eps \simeq \left(\frac1N\right)^{1/d}$, the term
$\frac1N \left(\frac{1}{\eps}\right)^{d-1}$ in \eqref{eq:PL} has order
$\left(\frac{1}{N}\right)^{1/d}$. On the other hand, as recalled in
\cref{rmk:rateminF}, $\min F_N \lesssim
\left(\frac{1}{N}\right)^{2/d}$ when $d>2$. Thus, we do not expect
 \eqref{eq:PL} to be tight.

\paragraph{Convergence of a discrete gradient flow}
The  modified Polyak-\L ojasiewicz inequality
\eqref{eq:PL} suggests that the discrete gradient flow
\ref{eq:discr_flow} will bring us close to a point cloud with low
Wasserstein distance to $\rho$, provided that  can guarantee that the
the points clouds $Y^k$ remain far for generalized diagonal 
during the iterations. We prove in Lemma~\ref{lem:stretch} in \cref{app:gf}
that if $Y^{k+1} = Y^k - \tau_N \nabla F_N(Y^k)$ and $\tau_N \in (0,1)$, then
\begin{equation}
	\label{ineq:res_gw}
	\forall i\neq j, \quad \norm{y^{k+1}_i-y^{k+1}_j} \geq (1 - \tau_N) \nr{y^k_i - y^k_j}.
\end{equation}
We note that this inequality  ensures that $Y^k$ never touches the generalized diagonal $\Diag_N$, so that the
gradient $\nabla F_N(Y^k)$ is well-defined at each step. Combining
this inequality with \cref{thm:approx_bary}, one can actually prove
that if the points in the initial cloud $Y^0_N$ are not too close to
each other, then a few steps of gradient discrete gradient descent
leads to a discrete measure $Y^k_N$ that is close to the target
$\rho$. Precisely, we arrive at the following theorem, proved in \cref{app:gf}.

\begin{thm} \label{thm:gf}
  Let $0<\alpha<\frac{1}{d-1}-\frac{1}{d}$, $\epsilon_N\gtrsim
  N^{-\frac{1}{d}-\alpha}$, and $Y^0_N \in
  \Omega^N\setminus\Diag_{\eps_N}$.  Let $(Y_N^k)_k$ be the iterates
  of \eqref{eq:discr_flow} starting from $Y_N^0$ with timestep
  $0<\tau_N<1$. We assume that $\lim_{N\to\infty}\tau_N=0$ and we set
	$$ k_{N}
	= \Partentf{\frac{1}{d\tau_N}\ln(F_N(Y_N^0)N\epsilon_N^{d-1})}.$$
	Then,
        \begin{equation}\label{eq:cvgd}
	  \Wass_2^2\left(\rho,\delta_{Y_N^{k_N}}\right)=O_{N\to\infty} \left( \Wass_2^2\left(\rho,\delta_{Y_N^0}\right)^{1-\frac{1}{d}} .N^{\frac{-1}{d^2}+\alpha\left(1-\frac 1d\right)} \right).
        \end{equation}
\end{thm}

\begin{rmk}\label{rmk:gf}
	Note that the exponential behavior implied by
        \ref{ineq:res_gw} and \cref{lem:stretch} is coherent with the
        estimates that are known in the absolutely continuous setting
        for the continuous gradient flow. When transitioning from
        discrete measures to probability densities, lower bounds on
        the distance between points become upper bounds on the
        density. The gradient flow $\dot{\mu}_t = \frac{1}{2}
        \nabla_\mu \Wass_2^2(\rho,\mu_t)$ has an explicit solution
        $\mu_t=\sigma_{1-e^{-t}}$, where $\sigma$ is a constant-speed
        geodesic in the Wasserstein space with $\sigma_0=\mu_0$ and
        $\sigma_1=\rho$. In this case, a simple adaptation of the
        estimates in Theorem 2 in \cite{santambrogio2009absolute}
        shows the bound $\nr{\mu_t}_{\LL^\infty}\leq e^{td}
        \nr{\mu_0}_{\LL^\infty}.$ Still in this absolutely continous
        setting, it is possible to remove the exponential growth if
        the target density is also bounded, as a consequence of
        \emph{displacement convexity} \cite[Theorem
          2.2]{mccann1997convexity}. There seems to be no discrete
        counterpart to this argument, explaining in part the
        discrepancy between the exponent of $N$ in \eqref{eq:cvgd}
        with the one obtained in \cref{cor:onestep}.
	%% Since there is no discrete analogue of displacement convexity, but In our case, it does not seem possible because McCann's proof relies on a convexity argument which does not hold in the discrete case.
	%% Here this is not
	%% possible for two reasons: on the one hand a suitable convexity lemma
	%% should be proven in this setting, which seems out of reach, and on
	%% the other hand the limit configuration as $t\to\infty$ is not known
	%% but is just a critical point in the optimal quantization of
	%% $\rho_0$, and could depend on the initial datum.
\end{rmk}

%\begin{proof}
%	This is obtained in a very similar fashion as \cref{lem:stretch}. For any $k\geq0$, the semi-concavity of $F_N$ yields the inequality:
%	\begin{equation*}
%		\begin{split}
%			F_N(Y_N^{k+1})-\frac{\norm{Y_N^{k+1}}^2}{2N}-\left(F_N(Y_N^k)-\frac{\norm{Y_N^k}^2}{2N}\right)\leq& \sca{-\B_N^k}{Y_N^{k+1})-Y_N^k}
%		\end{split}
%	\end{equation*}
%with $B_N^k:=B_N(Y_N^k)$ in accordance with the previous proof.
%
%Rearranging the terms, 
%	\begin{equation*}
%		\begin{split}
%			F_N(Y_N^{k+1})-F_N(Y_N^k)\leq&-\tau_N(2-\tau_N)\norm{B_N^k-Y_N^k}^2\\
%			=&-N\tau_N(2-\tau_N)\Wass_2^2(\delta_{B_N^k},\delta_{Y_N^k})\\
%			\leq&\tau_N(2-\tau_N)\left(-\frac{N}{2} \Wass_2^2(\delta_{Y_N^k},\rho)+ N\Wass_2^2(\rho,\delta_{B_N^k})\right)
%		\end{split}
%	\end{equation*}
%	by applying first the triangle inequality to $\Wass_2(\delta_{B_N^k},\delta_{Y_N^k})$ and then Cauchy-Schwartz's inequality. Using \cref{thm:approx_bary}, this yields:
%	\begin{equation*}
%		\begin{split}
%			F_N(Y_N^{k+1})\leq & (1-\frac{\tau_N}{2}(2-\tau_N)) F_N(Y_N^k)+2C_{d,\Omega}\tau_N(2-\tau_N)\epsilon_N^{1-d}(1-\tau_N)^{k(1-d)}\\
%			\leq & \eta_N F_N(Y_N^k)+2C_{d,\Omega}(1-\eta_N)\epsilon_N^{1-d}A_N^{k}.
%		\end{split}
%	\end{equation*}
%	and we simply iterate on $k$ to end up with the bound claimed in  \cref{lem:gf}.
%\end{proof}

%% file: sec4-numeric.tex
\begin{figure} \label{fig:Kanto_Lloyd}
	\centering\hfill
        \begin{minipage}{.56\textwidth}
	  {{\includegraphics[width=2.3cm]{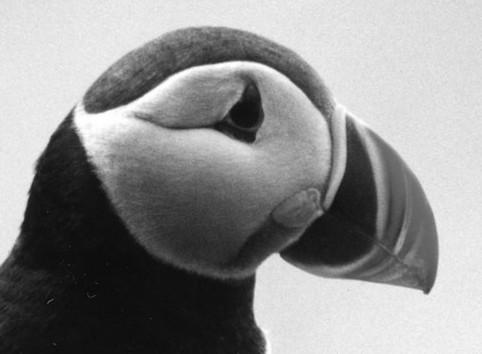}}}
	{{\includegraphics[width=2.5cm]{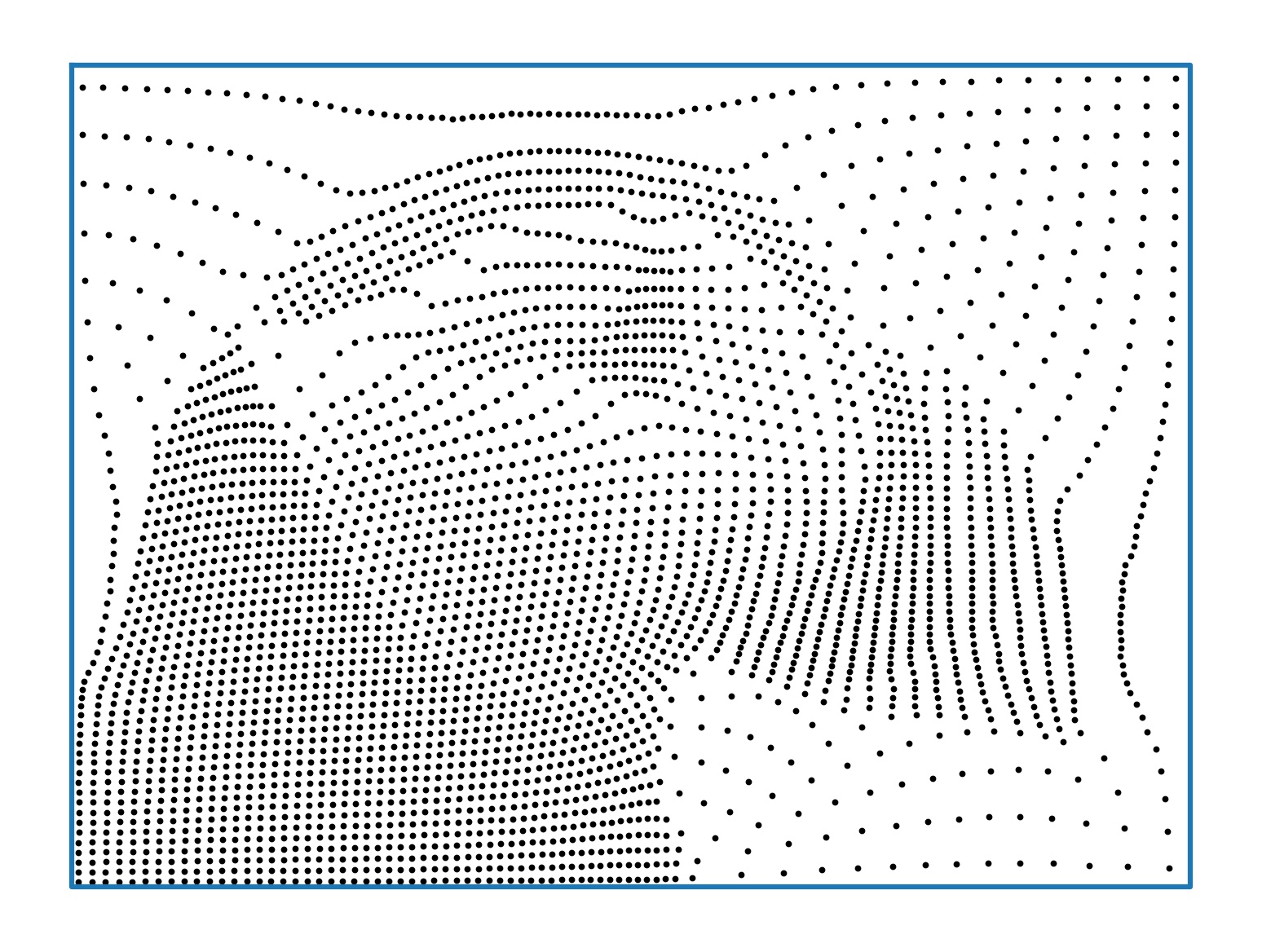}}}
	{{\includegraphics[width=2.5cm]{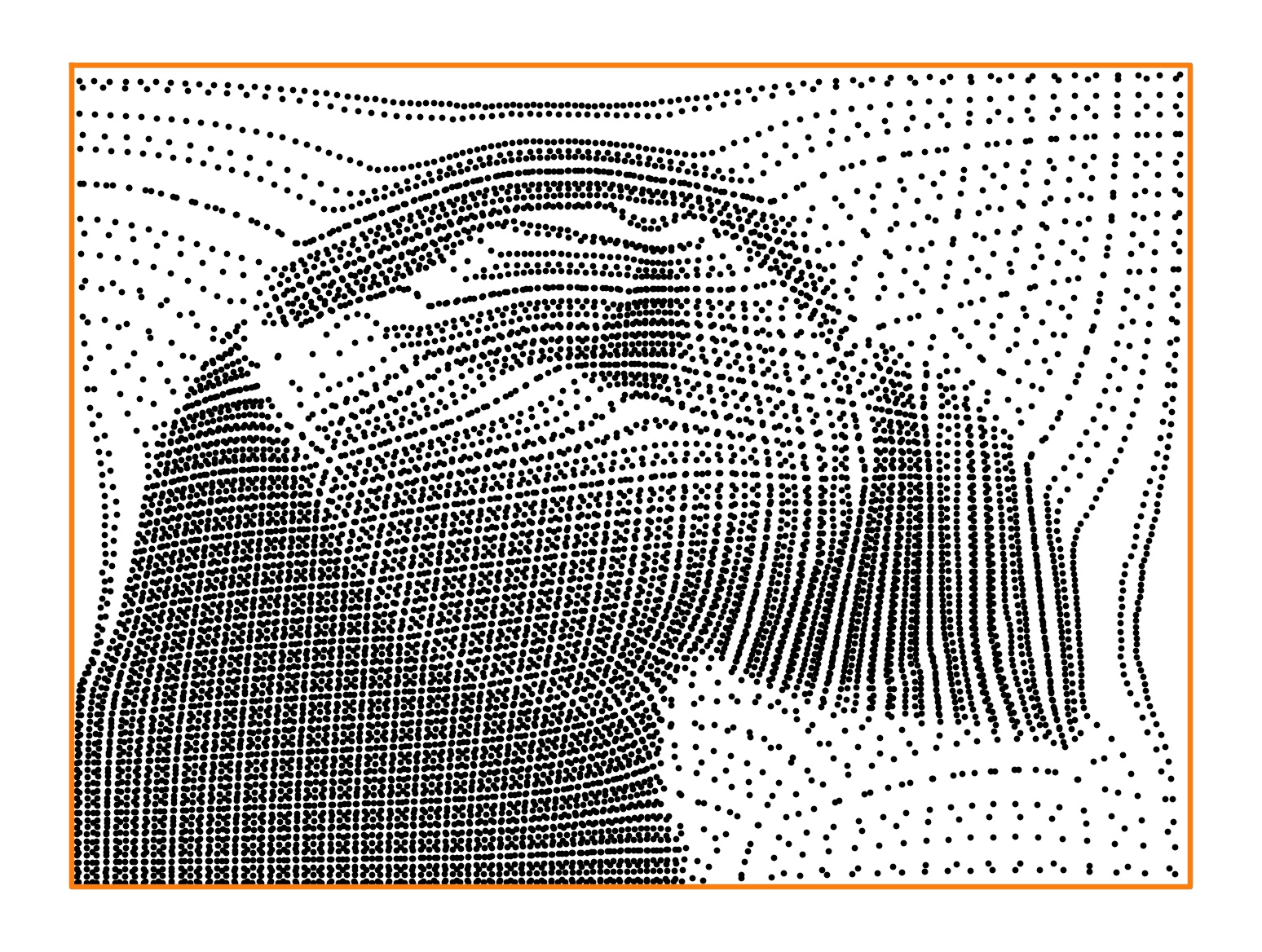}}}\\
        \phantom{{{\includegraphics[width=2.3cm]{figures/Puffin_gray.jpg}}}}
	{{\includegraphics[width=2.5cm]{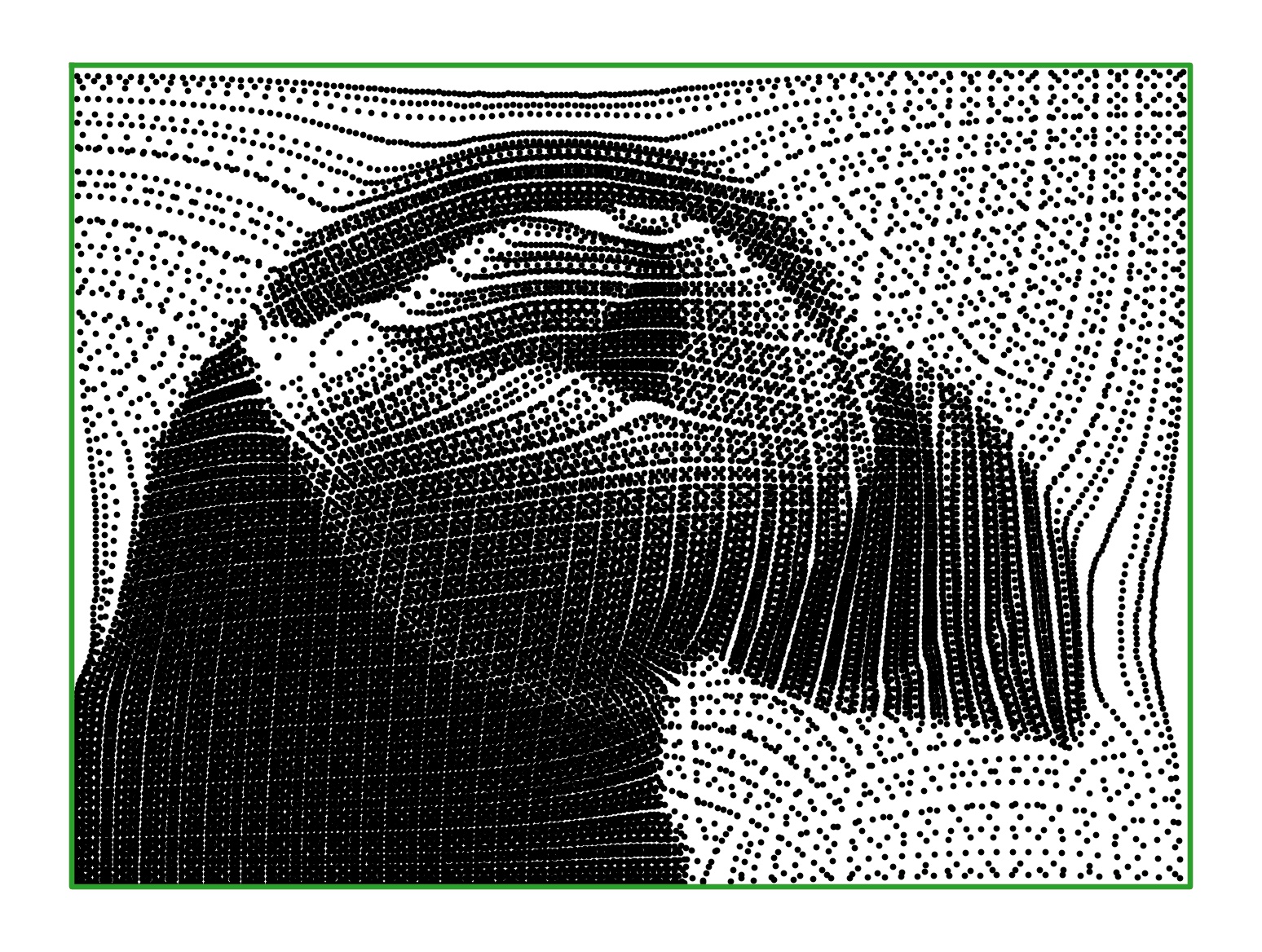}}}
	{{\includegraphics[width=2.5cm]{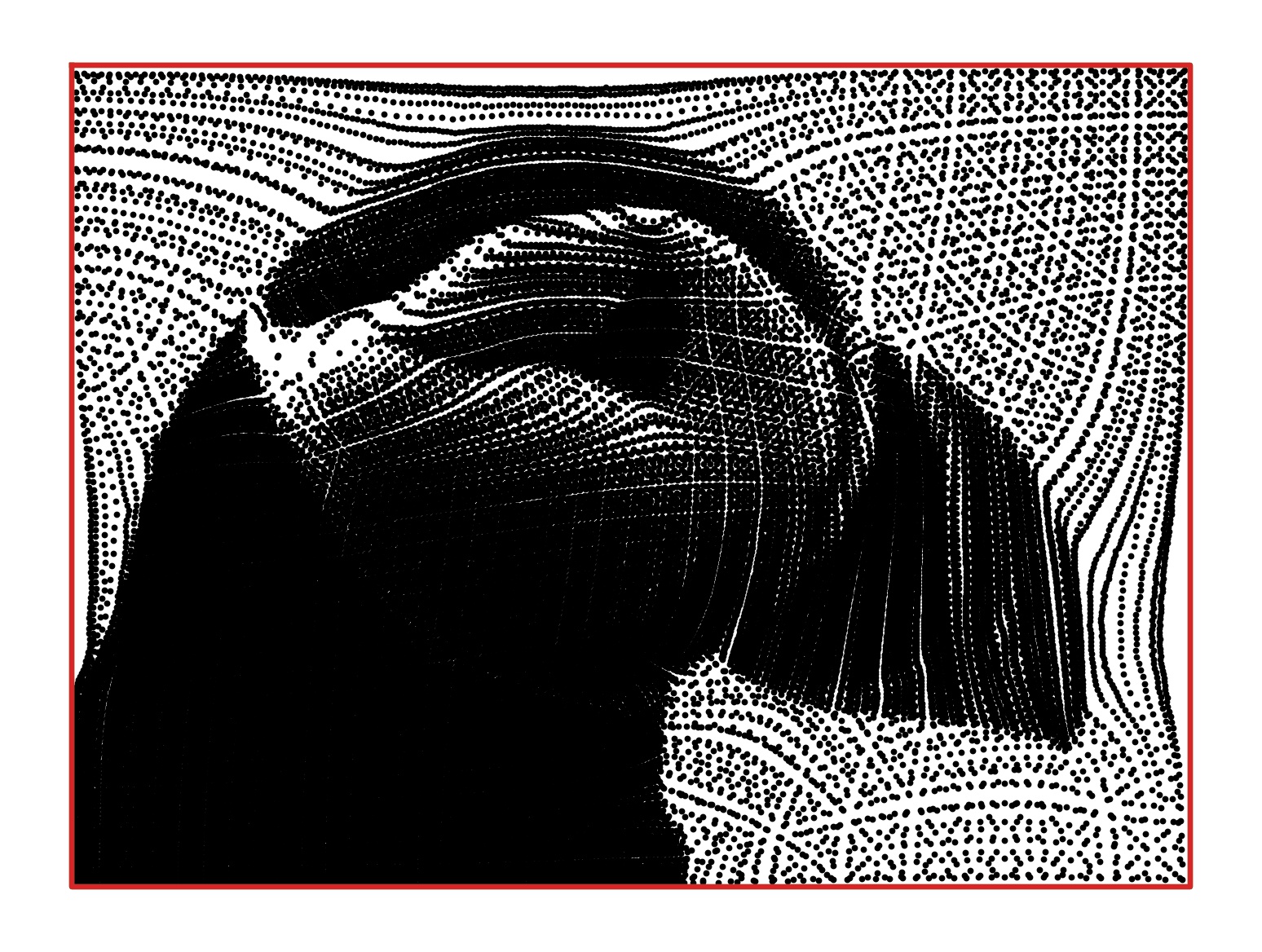}}}
        \end{minipage}
        \hfill\vline\hfill
        \begin{minipage}{.28\textwidth}
	  \includegraphics[width=\textwidth]{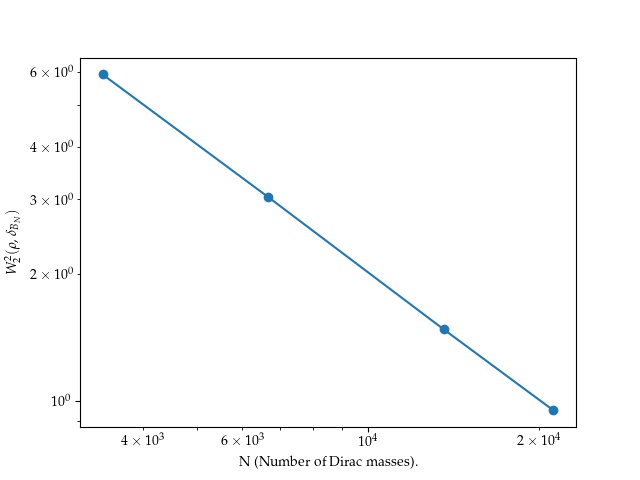}
        \end{minipage}
        \hfill
	\caption{\label{fig:Picture_Lloyd} Optimal quantization of a
          density $\rho$ corresponding to a gray-scale image
          (Wikimedia Commons, CC BY-SA 3.0). (Left) We display the
          point clouds obtained after one step of Lloyd's algorithm,
          starting from a regular grid of size $N
          \in\{3750,7350,15000,43350\}$.  (Right) Quantization error
          $W_2^2\left(\rho,\delta_{B_N}\right)$ as a function of N the
          number of points, showing that
          $W_2^2\left(\rho,\delta_{B_N}\right)\simeq N^{-1.00}$.}
\end{figure}

\section{Numerical results}
\label{sec:num}
In this section, we
report some experimental results in dimension $d=2$.

%\paragraph{Uniform}  In the case where $\rho$ is the Lebesgue measure on the domain
%$\Omega=[0,1]^2$, and if the point cloud $Y_N$ is sampled uniformly
%from $\Omega$, we experimentally recover the best possible convergence
%rate for a probability density on $\Rsp^2$, namely $F_N(B_N(Y_N)) \simeq N^{-\frac{2}{d}}.$ 

%% after only one step moving each point to the barycenter of its Power cell. This was the case whether we start from a random sample of point or a uniform grid, making it a case of little interest to us. However, it was unexpected to observe that this convergence rate for an initial uniform (in space) distribution of points is not observed for a seemingly equally regular case of a gaussian distribution.
\paragraph{Gray-scale image}

As we mentioned in the introduction, uniform optimal quantization
allows to sparsely represent a (gray scale) image via points,
clustered more closely in areas where the image is darker
\cite{de2012blue,balzer2009capacity}. On figure
\ref{fig:Picture_Lloyd}, we ploted the point clouds obtained after a
single Lloyd step toward the density representing the image on the
left (Puffin), starting from regular grids. The observed rate of
convergence, $N^{-1.00}$, is coherent with the theoretical estimate
$\log(N)/N$ of \cref{rmk:rateminF}.

%Finally, we conclude with the same study for one step of the Lloyd
%algorithm, starting from a uniform grid and converging toward a
%density representing the gray-scale picture of Kantorovich (see
%Figure~\ref{fig:unif_Kanto_bary}).  Starting from a choice of
%$N=961,~2500,~3600,~ 6400$ and $8100$ points $Y_N$ on a regular grid
%in the square delimiting our picture, we apply one step of Lloyd's
%algorithm \eqref{algo:lloyd2}. Again, the curve at the bottom of
%\ref{fig:unif_Kanto_bary} represents in a log-log scale the evolution
%of $F_N(B_N(Y_N))$ with respect to $N$, showing, as was the case for the
%Lebesgue measure a decrease rate close to $N^{-0.98}$. Observing the
%pictures above however, which correspond to the clouds of barycenters
%$B_N(Y_N)$ for $N=961,~3600$ and $8100$, we can observe some regular
%patterns reminiscent of the grid used as a seed for the Lloyd step,
%and which can be compared to \cref{fig:gaussian}.

\paragraph{Gaussian density with small variance} 
\label{rmk:gaussian}
We  now consider a toy model where we approximate a gaussian density
truncated to the unit square $\Omega=[0,1]^2$,
$\rho(x,y)=\frac{1}{Z}e^{-8((x-\frac12)^2+(y-\frac 12)^2)}$ where $Z$
is a normalization constant.  On the left column of this figure, the
initial point clouds $Y_N^0$ are randomly distributed in  $[0,1]^2$. The three point clouds represented
above are obtained after one step of Lloyd's algorithm
\eqref{algo:lloyd2}. The red curve displays in a log-log scale the
mean values of $F_N(B_N(Y_N))$ over a hundred random point clouds, for
 $N \in \{400,961,1600,2500\}$. In this case, we observe a
decrease rate $N^{-0.95}$ with respect to the number of points,
similar to the case of the gray scale picture.
          
However, an interesting phenomena occurs when the initial point cloud
$Y_N^0$ is aligned on a axis-aligned grid. The pictures in the right
column of \cref{fig:gaussian} where computed starting from such a grid
with $N \in \{400,961,1600,2500\}$ points. As in the randomly
initialized case, we represented the values of $F_N(B_N(Y_N))$ in log-log
scale. The corresponding discrete probability measure
$\delta_{B_N(Y_N)}$ seems to converge to $\rho$ as $N\to\infty$, but
with a much worse rate for these "low" values of $N$: $F_N(B_N(Y_N))
\simeq N^{-0.8}$.  In this specific setting, with a separable density
and an axis-aligned grid $Y_0$, the power cells are rectangles and a
single Lloyd step brings us to a critical point of $F_N$. Thanks to
this remark, it is possible to estimate the approximation error from the
one-dimensional case. In fact, Appendix \ref{app:gaussian_behaviour}
shows that for any $\delta\in(0,1)$, there exists variances $\sigma_N
= \sigma_N(\delta)$ such that the approximation error
$W_2^2(\rho_{\sigma_N},\delta_{B_N})$ is of order $N^{-\frac{2-\delta}{2}}$. On
the other hand, for a fixed $\sigma$, the approximation error
is of order $N^{-1}$, to be compared with the
bound $\log(N)/N$ for general measures.

%% , and we also show in
%% \cref{app:gaussian_behaviour}, using the separability of the transport
%% in this case, that the approximation error decays with a rate
%% $N^{-1}$, to be compared with the theoretical $\log(N)/N$ for general
%% measures in dimension 2.

%% file: appendix.tex
\section{Proof of \cref{prop:cvlloyd}}
Given $Y = (y_1,\hdots,y_N)\in(\Rsp^d)^N \setminus \Diag_N$, one has for any $i\in\{1,\hdots,N\}$,
   %% $$%x\begin{aligned}
$$  \begin{aligned}
  \int_{P_i(Y)} \norm{x-y_i}^2 \dd \rho(x)
  &= \int_{P_i(Y)} \norm{x-b_i(Y) + b_i(Y)  - y_i}^2 \dd\rho(x)\\
  &=     \int_{P_i(Y)} \norm{x - b_i(Y)}^2 \dd\rho(x) + \frac{1}{N} \nr{b_i(Y) - y_i}^2.
\end{aligned}
$$
Summing these equalities over $i$ and remarking that the map $T_Y$ defined  by $\left.T_Y\right|_{P_i(Y)} = y_i$ is an optimal transport map between $\rho$ and $\delta_Y$, we get
$$\begin{aligned}\frac1N \nr{B_N(Y) - Y}^2 &=  \Wass_2^2(\rho, y_i) -  \sum_i
  \int_{P_i(Y)} \norm{x - b_i(Y)}^2 \dd\rho(x) \\
  &\leq \Wass_2^2(\rho,\delta_Y) - \Wass_2^2(\rho,\delta_{B_N(Y)}).\end{aligned}$$
Thus, with $Y^{k+1} = B_N(Y^k)$, we have
$$ N\nr{\nabla F_N(Y^k)}^2 =\frac1N \nr{Y^{k+1} - Y^k}^2 \leq 2(F_N(Y^k) - F_N(Y^{k+1})).$$
This implies that the values of $ F_N(Y^k) $ are decreasing in $k$ and, since they are bounded from below, that $\nr{\nabla F_N(Y^k)}\to 0$ since $\sum_k  \nr{\nabla F_N(Y^k)}^2<+\infty$. The sequence $(Y^k)_k$ can be easily seen to be bounded, since $F_N(Y^k)$ is bounded, which implies a bound on the second moment of $\delta_{Y^k}$. 

For fixed $N$, since all atoms of $\delta_{Y^k}$ have mass $1/N$, this implies that all points $y_i^k$ belong to a same fixed compact ball. If $\rho$ itself is compactly supported, we can also prove that all points $Y^{k+1}=B_N(Y^k)$ are contained in a compact subset of $(\mathbb R^d)^N \setminus \Diag_N$, which means obtaining a lower bound on the distances $|b_i(Y)-b_j(Y)|$ for arbitrary $Y$. This lower bound can be obtained in the following way: since $\rho$ is absolutely continuous it is uniformly integrable which means that for every $\varepsilon>0$ there is $\delta=\delta(\varepsilon)>0$ such that for any set $A$ with Lebesgue measure $|A|<\delta$ we have $\rho(A)<\varepsilon$. We claim that we have $|b_i(Y)-b_j(Y)|\geq r:=(2R)^{1-d}\delta(\frac{1}{2N})$, where $R$ is such that $\rho$ is supported in a ball $B_R$ of radius $R$. Indeed, it is enough to prove that every barycenter $b_i(Y)$ is at distance at least $r/2$ from each face of the convex polytope $P_i(Y)$. Consider a face of such a polytope and suppose, by simplicity, that it lies on the hyperplane $\{x_d=0\}$ with the cell contained in $\{x_d\geq 0\}$. Let $s$ be such that $\rho(P_i(Y)\cap \{x_d> s\})=\rho(P_i(Y)\cap \{x_d< s\})=\frac{1}{2N}$. Then since the diameter of $P_i(Y)\cap B_R$ is smaller than $2R$, the Lebesgue measure of $P_i(Y)\cap \{x_d< s\}$ is bounded by $(2R)^{d-1}s$, which provides $s\geq r$ because of the definition of $r$. Since at least half of the mass (according to $\rho$) of the cell $P_i(Y)$ is above the level $x_d=s$ the $x_d$-coordinate of the barycenter is at least $r/2$. This shows that the barycenter lies at distance at least $r/2$ from each of its faces.

As a consequence, the iterations $Y^k$ of the Lloyd algorithm lie in a compact subset of $(\mathbb R^d)^N \setminus \Diag_N$, on which $F_N$ is $C^1$. This implies that any limit point must be a critical point. 

We do not discuss here whether the whole sequence converges or not, which seems to be a delicate matter even for fixed $N$. It is anyway possible to prove (but we do not develop the details here) that the set of limit points is a closed connected subet of $(\mathbb R^d)^N$ with empty interior, composed of critical points of $F_N$ all lying on a same level set of $F_N$.

\section{Proof of \cref{coro:proba}} \label{app:proba}
Given $Y = (y_1,\hdots,y_N) \in (\Rsp^d)^N$, we denote
$$ I_\eps(Y) = \{ i\in \{1,\hdots, N\} \mid \forall j\neq i, \|y_i -
y_j\|\geq \eps \}.$$ We call points $y_i$ such that $i\in I_\eps(Y)$
$\eps$-isolated, and points $y_i$ such that $i\not\in I_\eps(Y)$
$\eps$-connected.

\begin{lem} \label{lem:Diarmid}
 Let $X_1,\hdots,X_N$ be independent, $\Rsp^d$-valued, random
 variables. Then, there is a constant $C_d>0$ such that
 $$ \Prob(\{ \abs{\kappa(X_1,\hdots,X_N) - \Exp(\kappa)}\geq \eta \})  \leq \ee^{-N \eta^2/C_d}.$$
\end{lem}

\begin{proof}
This lemma is a consequence of McDiarmid's inequality. To apply this
inequality, we need evaluate the amplitude of variation of the
function $\kappa$ along changes of one of the points $x_i$. Denote
$c_d$ the maximum cardinal of a subset $S$ of the ball $B(0,\eps)$
such that the distance between any distinct points in $S$ is at least
$\eps$. By a scaling argument, one can check that $c_{d}$ does not, in
fact, depend on $\eps$. To evaluate
$$ \abs{\kappa(x_1,\hdots,x_i,\hdots,x_N) -
  \kappa(x_1,\hdots,\tilde{x}_i,\hdots,x_N)}, $$ we first note that at
most $c_d$ points may become $\eps$-isolated when removing $x_i$. To
prove this, we remark that if a point $x_j$ becomes $\eps$-isolated
when $x_i$ is removed, this means that $\|{x_i - x_j}\| \leq \eps$ and
$\|{x_j - x_k}\| > \eps$ for all $k\not\in\{i,j\}$.  The number of
such $j$ is bounded by $c_d$. Symmetrically, there may be at most
$c_d$ points becoming $\eps$-connected under addition of
$\hat{x}_i$. Finally, the point $x_i$ itself may change status from
$\eps$-isolated to $\eps$-connected. To summarize, we obtain that with
$C_d = 2c_d + 1$,
  $$\abs{\kappa(x_1,\hdots,x_i,\hdots,x_N) -
    \kappa(x_1,\hdots,\tilde{x}_i,\hdots,x_N)} \leq \frac{1}{N} C_d.$$
The conclusion then directly follows from McDiarmid's inequality.
\end{proof}

\begin{lem}
	\label{lem:exp_kappa}
  Let $\sigma \in \LL^\infty(\Rsp^d)$ be a probability density and let
  $X_1,\hdots,X_N$ be i.i.d. random variables with distribution
  $\sigma$. Then,
  $$ \Exp(\kappa(X_1,\hdots,X_N)) \geq (1 - \|{\sigma}\|_{\LL^\infty} \omega_d \epsilon^d)^{N-1}.$$
\end{lem}

\begin{proof}
  The probability that a point $X_i$ belongs to the ball $B(X_j,\eps)$
  for some $j\neq i$ can be bounded from above by $\sigma(B(X_j,\eps))
  \leq \|{\sigma}\|_{\LL^\infty} \omega_d \epsilon^d$, where $\omega_d$
  is the volume of the $d$-dimensional unit ball. Thus, the
  probability that $X_i$ is $\eps$-isolated is larger than
  $$ (1 - \|{\sigma}\|_{\LL^\infty} \omega_d \epsilon^d)^{N-1}. $$
  We conclude by noting that
  \begin{equation*}
    \Exp(\kappa(X_1,\hdots,X_N)) = \frac{1}{N} \sum_{1\leq i\leq N} \Prob(X_i \hbox{ is $\eps$-isolated}).
  \qedhere
  \end{equation*}
\end{proof}

\begin{proof}[Proof of \cref{coro:proba}]
	We apply the previous \cref{lem:exp_kappa} with $\epsilon_N = N^{-\frac{1}{\beta}}$ and $\beta=d-\frac{1}{2}$.
	The expectation of $\kappa(X_1,\dots,X_N)$ is lower bounded by:	\begin{equation*}
		\begin{split}
			\Exp(\kappa(X_1,\dots,X_N))\geq&\left(1 - N^{-\frac{d}{\beta}}\|{\sigma}\|_{\LL^\infty} \omega_d\right)^{N-1}\\
			\geq&1-CN^{1-\frac{d}{\beta}}
		\end{split}
	\end{equation*} 
	for large $N$, since $\beta<d$.
	By \cref{lem:Diarmid}, for any $\eta>0$,
	$$ \Prob(\kappa(X_1,\hdots,X_N)\geq1-CN^{1-\frac{d}{\beta}}-\eta )  \geq 1-e^{-K N\eta^2},$$
	for constants $C,K>0$ depending only on $\|{\sigma}\|_{\LL^\infty}$ and $d$. We choose $\eta=N^{-\frac{1}{2d-1}}$, so that $\eta$ is of the same order as $N^{1-\frac{d}{\beta}}$ since $1-\frac{d}{\beta}=-\frac{1}{2d-1}$.Thus, for a slightly different $C$,
	$$ \Prob(\kappa(X_1,\hdots,X_N)\geq 1-C\eta )  \geq 1-\ee^{-K N\eta^2}.$$
	Now, for $\omega_1,\dots,\omega_N$ such that $$\kappa(X_1(\omega_1),\hdots,X_N(\omega_N))\geq 1-C\eta,$$
	\cref{thm:approx_bary} yields:
	\begin{equation*}
		W_2^2\left(\delta_{B_N(X(\omega))},\rho\right)\lesssim\frac{N^{\frac{d-1}{\beta}}}{N}+\eta
		\lesssim N^{-\frac{1}{2d-1}}
	\end{equation*}

	and such a disposition happens with probability at least
	\begin{equation*}
		1-\ee^{-KN\eta^2}=1-\ee^{-KN^{\frac{2d-3}{2d-1}}}. \qedhere
	\end{equation*}
\end{proof}

\section{Proof of \cref{cor:PL}}
We first note that by Proposition~\ref{prop:gradF}, we have
$\nr{\nabla F_N(Y)}^2 =\frac{1}{N^2} \norm{B_N(Y)-Y}^2 $. We then use $\Wass^2_2(\delta_{B_N(Y)},\delta_Y)\leq \frac{1}{N} \norm{B_N(Y)-Y}^2 $ and 
$$
\begin{aligned}
  \Wass_2^2(\rho,\delta_Y) &\leq 2\big(\Wass_2^2(\rho,\delta_{B_N(Y)}) + 2N\nr{\nabla F_N(Y)}^2 .\end{aligned}$$
Thus, using \cref{thm:approx_bary} to bound
$\Wass_2^2(\rho,\delta_{B_N(Y)})$ from above, we get the desired
result.

\section{Proof of \cref{thm:gf}}
\label{app:gf}

\begin{lem} 
  \label{lem:stretch}
  Let $Y^0\in(\Rsp^d)^N\setminus \Diag_{N,\eps_N}$ for some $\eps_N>0$. Then, the iterates
  $(Y^k)_{k\geq 0}$ of \eqref{eq:discr_flow} satisfy for every $k\geq
  0$, and for every $i\neq j$
	\begin{equation}
		\label{eq:reverse_gronw}
		\norm{y_i^k-y_j^k}\geq(1-\tau_N)^{k}\epsilon_N
	\end{equation}
\end{lem}
\begin{proof}
	We consider the distance between two trajectories after $k$ iterations: $e_k=\norm{y_i^k-y_j^k}.$ Assuming that $e_k>0$, the convexity of the norm immediately gives us:
	\begin{align*}
		e_{k+1}-e_k \geq& \sca{\frac{y_i^k-y_j^k}{\norm{y_i^k-y_j^k}}}{y_i^{k+1}-y_j^{k+1}-\left(y_i^{k}-y_j^{k}\right)}\\ 
		=&
		\tau_N\sca{\frac{y_i^k-y_j^k}{\norm{y_i^k-y_j^k}}}{b_i^k-b_j^k}
		-\tau_N\norm{y_i^k-y_j^k}
	\end{align*}
	where we denoted $b_i^k:=b_i(Y_N^k)$ the barycenter of the $i$th
	Power cell $P_i(Y_N^k)$ in the tesselation associated with the point cloud
	$Y_N^k$.  Since each barycenter $b_i^k$ lies in its
	corresponding Power cell, the scalar product $\sca{y_i^k-y_j^k}{b_i^k-b_j^k}$ is non-negative:
	Indeed, for any $i\neq j$,
	\begin{equation*}
		\norm{y_i^k - b_i^k}^2 - \norm{y_j^k - b_i^k}^2 \leq \phi_i^k-\phi_j^k
	\end{equation*}

	Summing this inequality with the same inequality with the roles of $i$ and $j$ reversed, we obtain:
	$$\sca{y_i^k-y_j^k}{b_i^k-b_j^k}\geq 0$$
	thus giving us the geometric inequality $e_{k+1}\geq (1-\tau_N)e_k$. Since $Y_N^0$ was chosen in  $\Omega^N\setminus\Diag_{N,\eps_N}$, this yields $e_k\geq (1-\tau_N)^k e_0$ and inequality \ref{eq:reverse_gronw}.
\end{proof}

\begin{lem}
	For any $k\geq0$
	\label{lem:gf}
	\begin{equation}
		\label{ineq:prec_gf}
		F_N(Y_N^k)\leq F_N(Y_N^0)\eta_N^k+2C_{d,\Omega}(1-\eta_N)\frac{\epsilon_N^{1-d}}{N}\frac{A_N^k-\eta_N^k}{A_N-\eta_N},
	\end{equation}
	where we denote  $\eta_N=1-\frac{\tau_N}{2}(2-\tau_N)$ and $A_N=(1-\tau_N)^{1-d}$.
\end{lem}

\begin{proof}
	This is obtained in a very similar fashion as \cref{lem:stretch}. For any $k\geq0$, the semi-concavity of $F_N$ yields the inequality:
	\begin{equation*}
		\begin{split}
			F_N(Y_N^{k+1})-\frac{\norm{Y_N^{k+1}}^2}{2N}-\left(F_N(Y_N^k)-\frac{\norm{Y_N^k}^2}{2N}\right)\leq& \sca{-\frac{\B_N^k}{N}}{Y_N^{k+1}-Y_N^k}
		\end{split}
	\end{equation*}
with $B_N^k:=B_N(Y_N^k)$ in accordance with the previous proof.

Rearranging the terms, 
	\begin{equation*}
		\begin{split}
			F_N(Y_N^{k+1})-F_N(Y_N^k)\leq&-\tau_N(1-\frac{\tau_N}{2})\frac{\norm{B_N^k-Y_N^k}^2}{N}\\
			=&-\tau_N(1-\frac{\tau_N}{2})\Wass_2^2(\delta_{B_N^k},\delta_{Y_N^k})\\
			\leq&\tau_N(1-\frac{\tau_N}{2})\left(-\frac{1}{2} \Wass_2^2(\delta_{Y_N^k},\rho)+ \Wass_2^2(\rho,\delta_{B_N^k})\right)
		\end{split}
	\end{equation*}
	by applying first the triangle inequality to $\Wass_2(\delta_{B_N^k},\delta_{Y_N^k})$ and then Cauchy-Schwartz's inequality. Using \cref{thm:approx_bary}, this yields:
	\begin{equation*}
		\begin{split}
			F_N(Y_N^{k+1})\leq & (1-\frac{\tau_N}{2}(2-\tau_N)) F_N(Y_N^k)+2C_{d,\Omega}\tau_N(2-\tau_N)\frac{\epsilon_N^{1-d}}{N}(1-\tau_N)^{k(1-d)}\\
			\leq & \eta_N F_N(Y_N^k)+2C_{d,\Omega}(1-\eta_N)\frac{\epsilon_N^{1-d}}{N}A_N^{k}.
		\end{split}
	\end{equation*}
	and we simply iterate on $k$ to end up with the bound claimed in  \cref{lem:gf}.
\end{proof}

\begin{proof}[Proof of \cref{thm:gf}]
	To conclude, we simply make (order 1) expansions of the terms in \ref{ineq:prec_gf}. The definition of $k_N$ in \cref{thm:gf}, although convoluted, was made so that both terms in the right-hand side of this inequality, $F_N(Y_N^0)\eta_N^{k_N}$ and $(1-\eta_N)\frac{\epsilon_N^{1-d}}{N}\frac{A_N^{k_N}-\eta_N^{k_N}}{A_N-\eta_N}$ have the same asymptotic decay to $0$ (as $N\to+\infty$):	With the notations of the previous proposition, we have for fixed $N$:
		\begin{equation}
			\label{eq:expl_bound}
			\Wass_2^2\left(\rho,\delta_{Y_N^{k_N}}\right)\leq \Wass_2^2\left(\rho,\delta_{Y_N^0}\right)\eta_N^{k_N}+2C_{d,\Omega}\frac{(1-\eta_N)}{A-\eta_N}\frac{A^{k_N}-\eta_N^{k_N}}{N\epsilon_N^{d-1}}
		\end{equation}
		
		We make use here of the notation from \cref{sec:gf}: 
		
		$$T_N=k_N\tau_N=\Partentf{\frac{1}{d}\ln(F_N(Y_N^0)N\epsilon_N^{d-1})}$$
		to clear this expression a bit, and, because of the assumption $\lim_{N\to\infty}\tau_N=0$, we may write:
		\begin{equation*}
			\frac{A^{k_N}-\eta^{k_N}}{N\eps_N^{d-1}}=\frac{\ee^{(d-1)T_N}}{N\eps_N^{d-1}}+o_{N\to\infty}\left(\frac{T_N}{(N\epsilon_N^{d-1})^{\frac{1}{d}}}\right)
		\end{equation*}
		as well as $\eta^{k_N}=\ee^{- T_N}+o_{N\to\infty}\left(\frac{T_N}{(N\epsilon_N^{d-1})^{\frac{1}{d}}}\right)$, and substituting $T_N$,
		\begin{equation*}
			\begin{split}
				\Wass_2^2\left(\rho,\delta_{Y_N^{k_N}}\right)\lesssim& \frac{\Wass_2^2\left(\rho,\delta_{Y_N^0}\right)^\frac{d-1}{d}} {\left(N\epsilon_N^{d-1}\right)^{\frac 1d}}+o_{N\to\infty}\left(\frac{T_N}{(N\epsilon_N^{d-1})^{\frac{1}{d}}}\right)\\
				\lesssim&\Wass_2^2\left(\rho,\delta_{Y_N^0}\right)^{1-\frac{1}{d}} N^{\frac{-1}{d^2}+\alpha\left(1-\frac 1d\right)}
				\qedhere
			\end{split}
		\end{equation*}
\end{proof}

\section{Case of a low variance Gaussian in \cref{rmk:gaussian}}
\label{app:gaussian_behaviour}
	Here, we consider $\rho_\sigma$ the probability measure
        obtained by truncating and renormalizing a centered normal
        distribution with variance $\sigma$ to the segment
        $[-1,1]$. We first show that for any $N\in\N$ and
        $\delta\in(0,1)$, we can find a small $\sigma_{N,\delta}$ such
        that the Wasserstein distance beween
        $\rho_{\sigma_{N,\delta}}$ and its best $N$-points
        approximation of is at least $CN^{-(2-\delta)}$.

        %% The transport problem featured in \cref{rmk:gaussian}, paragraph 2 can be decomposed into two 1-dimensional approximation of gaussan distributions, by $\sqrt{N}$ points instead of $N$, meaning both that the Lloyd step will bring us to a cloud aligned on a regular grid, and that the approximation error is simply the sum of the 1-D errors. Therefore, the best approximation of the 2-D measure $\rho_{\sigma_{N,\delta}}\otimes\rho_{\sigma_{N,\delta}}$, with $N$ points aligned on a regular grid, will be at a squared distance at least $CN^{-(2-\delta)/2}\gg N^{-1}$, thus explaining the numerical observations.

    \begin{prop}
    	\label{prop:gaussian}
    	For any $\sigma>0$, consider $\rho_\sigma \eqdef m_\sigma \ee^{-\frac{\abs{x}^2}{2\sigma^2}} \mathds{1}_{[-1;1]} dx$ the truncated centered Gaussian density, where $m_\sigma$ is taken so that $\rho_\sigma$ has unit mass. Then, for every $\delta\in(0,1)$, there exists a constant $C>0$ and a sequence of variances $(\sigma_N)_{N\in\N}$ such that
    	$$ \forall Y\in(\R^d)^N\setminus\Diag_N,\quad \Wass_2^2\left(\delta_{B_N(Y)},\rho_{\sigma_N}\right)\geq CN^{-(2-\delta)}$$
    \end{prop}
    From the proof, one can see that the dependence of $\sigma_N$ on $N$ is logarithmic.
	\begin{proof}

	We denote
	$g:x\in\R\mapsto\frac{1}{\sqrt{2\pi}}\ee^{-\frac{\abs{x}^2}{2}}$ the
	density of the centered Gaussian distribution and $F_g$ its cumulative
	distribution function, so that
	\begin{equation}
		\begin{split}
			\label{eq:msig}
			m_\sigma^{-1} = \int_{-1}^{1}\ee^{-\frac{\abs{x}^2}{2\sigma^2}}dx = \sigma \sqrt{2\pi} \int_{-1/\sigma}^{1/\sigma} g(y) dy = \sqrt{2\pi} \sigma (F_g(1/\sigma) -  F_g(-1/\sigma))
		\end{split}
	\end{equation}
	Note that, whenever $\sigma\to 0$, we have $\sigma m_\sigma\to \sqrt{2\pi}$. We denote by $F_\sigma:[-1,1]\to[0,1]$ the cumulative distribution
	function of $\rho_\sigma$.
	Given any point cloud $Y = (y_1,\hdots,\leq y_N)$ such that $y_1\leq
	\hdots\leq y_N$, the Power cells $P_i(Y)$ is simply the segment
	$$P_i(Y) = [F_\sigma^{-1}(i/N), F_\sigma^{-1}((i+1)/N)]. $$ Since these
	segments do not depend on $Y$, we will denote them $(P_i)_{1\leq
		i\leq N}$.  Finally, defining $b_i = N \int_{P_i} x \dd
	\rho_\sigma(x)$ as the barycenter of the $i$th power cell and
	$\delta_B = \frac{1}{N}\sum_i \delta_{b_i}$, we have
	\begin{equation}\label{W2below}
		\begin{aligned}
			\Wass_2^2(\delta_{B}, \rho_\sigma) &= \sum_{i=1}^N \int_{P_i} (x - b_i)^2 \dd\rho_\sigma(x)\\
			&\geq \rho_\sigma(-1) \sum_{i=1}^N \int_{P_i} (x - b_i)^2 \dd x \\
			&\geq C \rho_\sigma(-1) \sum_{i=1}^N (F_\sigma^{-1}((i+1)/N) - F_\sigma^{-1}(i/N))^3,
		\end{aligned}
	\end{equation}
	where we used that $\rho_\sigma$ attains its minimum at $\pm 1$ to get the first inequality.
	%$$F(1/\sigma)-F(-1/\sigma)=\frac{1}{\sqrt{2\pi}\sigma m_\sigma}$$
	We now wish to provide an approximation for $F_\sigma^{-1}(t)$, $t\in[0,1]$. We first note,  using Taylor's formula, that we have
	$$ \begin{aligned}
		F_\sigma^{-1}(t) &=
		\sigma F_g^{-1}\left(F_g\left(\frac{-1}{\sigma}\right)+t\left[F_g\left(\frac 1\sigma\right) - F_g\left(\frac {-1}\sigma\right)\right]\right) \\
		&= \sigma F_g^{-1}\left(F_g\left(\frac{-1}{\sigma}\right)+ \frac{t}{\sqrt{2\pi} \sigma m_\sigma}\right)\\
		&= -1 + \sigma (F_g^{-1})'\left(F_g\left(\frac{-1}{\sigma}\right)\right) \frac{t}{\sqrt{2\pi} \sigma m_\sigma}
		+ \frac{\sigma}{2} (F_g^{-1})''(s) \frac{t^2}{2\pi \sigma^2 m_\sigma^2}
	\end{aligned}
	$$
	for some $s\in [F_g(-\frac 1\sigma), F_g(-\frac 1\sigma) + t (F_g(\frac 1\sigma) - F_g(-\frac 1\sigma))]$.
	But,
	$$(F_g^{-1})'(t)=\frac{1}{g\circ F_g^{-1}(t)}=\sqrt{2\pi}\ee^{\frac{\abs{F_g^{-1}(t)}^2}{2}},$$
	$$(F_g^{-1})''(t)=-\frac{g'\circ F_g^{-1}(t)}{\left(g\circ F_g^{-1}(t)\right)^3}=2\pi F_g^{-1}(t)\ee^{\abs{F_g^{-1}(t)}^2},$$
	and we see that 
	$$  \abs{F_\sigma^{-1}(t) - \left(-1+\frac{t}{ m_\sigma}\ee^{\frac{1}{2\sigma^2}}\right)} \leq \ee^{\frac{1}{\sigma^2}}\frac{t^2}{2\sigma^2 m_\sigma^2} $$

	Therefore, if we denote $\varepsilon(\sigma,t)$ the second-order error in the above formula, i.e.  $\varepsilon(\sigma,t)=\ee^{\frac{1}{\sigma^2}}\frac{t^2}{2\sigma^2 m_\sigma^2} $, the size of the first Power cell $P_0(Y)$ is of order:
	$$F_\sigma^{-1}(1/N) - F_\sigma^{-1}(0)=\frac{1}{N m_\sigma} \ee^{\frac{1}{2\sigma^2}}+O\left(\varepsilon\left(\sigma,\frac{1}{N}\right)\right).$$
	We will choose $\sigma_N$ depending on $N$ in order for the first term in the left-hand side to dominate the second one:
	\begin{equation}\label{a-bbb}
		\varepsilon\left(\sigma_N,\frac{1}{N}\right)=o\left(\frac{1}{N m_\sigma} \ee^{\frac{1}{2\sigma^2}}\right).
	\end{equation}
	In this way, we have
	\begin{equation}
		\label{eq:first}
		\begin{split}
			(F_\sigma^{-1}(1/N) - F_\sigma^{-1}(0))^3\rho_\sigma(-1)\geq &c\frac{1}{N^3 m_\sigma^3}\ee^{\frac{3}{2\sigma^2}}m_\sigma\ee^{-\frac{1}{2\sigma^2}}\\
			=&c\frac{1}{N^3m_\sigma^2}\ee^{\frac{1}{\sigma^2}}.
		\end{split}
	\end{equation}
	
	We now choose $\sigma=\sigma_{N}$ such that $\ee^{\frac{1}{2\sigma^2}}=N^\alpha$ for an exponent $\alpha$ to be chosen. We need $\alpha>0$ so that $\sigma_N\to 0$. This last condition and \eqref{eq:msig} implies that $m_{\sigma_{N}}$ is of order $\sqrt{\log N}$. This means that the condition \eqref{a-bbb} is satisfied if $\alpha<1$ and $N$ large enough.
	
	The sum in \eqref{W2below} is lower bounded by its first term, \eqref{eq:first}, and we get
	$$\Wass_2^2(\delta_{B}, \rho_\sigma) \geq c \frac{1}{N^3m_{\sigma_N}^2}\ee^{\frac{1}{\sigma_N^2}}\geq C\left(\frac{N^{2\alpha-3}}{\ln(N)}\right)$$
	for some constant $C>0$, since $\sigma$ depends logarithmically on $N$. Finally, if we want this last expression to be larger than $N^{-(2-\delta)}$ we can take for instance $2\alpha>1+\delta$ and $N$ large enough.
	\end{proof}
	
	The following corollary, whose proof can just be obtained by adapting the above proof to a simple multi-dimensional setting where measures and cells ``factorize'' according to the components, confirms the facts observed in the numerical section (Section~\ref{sec:num}), and the sharpness of our result (Remark~\ref{rmk:optimbeta}). 
        \begin{cor}\label{cor:gaussian}
	  Fix  $\delta\in(0,1)$. Given any $n\in\N$, consider an axis-aligned discrete grid of the form $Z_N = Y_1 \times \hdots \times Y_d$ in $\Rsp^d$, with $N = \Card(Z_N) = n^d$, where each $Y_j$ is a subset of $\Rsp$ with cardinal $n$. Finally, define $\sigma_N := \sigma_{n,\delta}$ as in \cref{prop:gaussian} Then we have
	  $$\Wass_2^2(\delta_{B_N(Z_N)},\rho_{\sigma_N}\otimes\dots\otimes\rho_{\sigma_N})\geq CN^{-\frac {(2-\delta)}{d}},$$
	  where the constant $C$ is independent of $N$.
	\end{cor}

%% file: optquant.bbl
\begin{thebibliography}{10}

\bibitem{ambrosio2008gradient}
Luigi Ambrosio, Nicola Gigli, and Giuseppe Savar{\'e}.
\newblock {\em Gradient flows: in metric spaces and in the space of probability
  measures}.
\newblock Springer Science \& Business Media, 2008.

\bibitem{arjovsky2017wasserstein}
Martin Arjovsky, Soumith Chintala, and L{\'e}on Bottou.
\newblock Wasserstein generative adversarial networks.
\newblock In {\em International conference on machine learning}, pages
  214--223. PMLR, 2017.

\bibitem{balzer2009capacity}
Michael Balzer, Thomas Schl{\"o}mer, and Oliver Deussen.
\newblock Capacity-constrained point distributions: a variant of lloyd's
  method.
\newblock {\em ACM Transactions on Graphics (TOG)}, 28(3):1--8, 2009.

\bibitem{de2012blue}
Fernando De~Goes, Katherine Breeden, Victor Ostromoukhov, and Mathieu Desbrun.
\newblock Blue noise through optimal transport.
\newblock {\em ACM Transactions on Graphics (TOG)}, 31(6):1--11, 2012.

\bibitem{du2006convergence}
Qiang Du, Maria Emelianenko, and Lili Ju.
\newblock Convergence of the lloyd algorithm for computing centroidal voronoi
  tessellations.
\newblock {\em SIAM journal on numerical analysis}, 44(1):102--119, 2006.

\bibitem{engquist2016optimal}
Bjorn Engquist, Brittany~D Froese, and Yunan Yang.
\newblock Optimal transport for seismic full waveform inversion.
\newblock {\em Communications in Mathematical Sciences}, 14(8):2309--2330,
  2016.

\bibitem{feydy2017optimal}
Jean Feydy, Benjamin Charlier, Fran{\c{c}}ois-Xavier Vialard, and Gabriel
  Peyr{\'e}.
\newblock Optimal transport for diffeomorphic registration.
\newblock In {\em International Conference on Medical Image Computing and
  Computer-Assisted Intervention}, pages 291--299. Springer, 2017.

\bibitem{genevay2018learning}
Aude Genevay, Gabriel Peyr{\'e}, and Marco Cuturi.
\newblock Learning generative models with {S}inkhorn divergences.
\newblock In {\em International Conference on Artificial Intelligence and
  Statistics}, pages 1608--1617. PMLR, 2018.

\bibitem{goes2014weighted}
Fernando~de Goes, Pooran Memari, Patrick Mullen, and Mathieu Desbrun.
\newblock Weighted triangulations for geometry processing.
\newblock {\em ACM Transactions on Graphics (TOG)}, 33(3):1--13, 2014.

\bibitem{graf2007foundations}
Siegfried Graf and Harald Luschgy.
\newblock {\em Foundations of quantization for probability distributions}.
\newblock Springer, 2007.

\bibitem{KitMerThi17}
Jun Kitagawa, Quentin M{\'e}rigot, and Boris Thibert.
\newblock Convergence of a newton algorithm for semi-discrete optimal
  transport.
\newblock {\em Journal of the European Mathematical Society}, 21(9):2603--2651,
  2019.

\bibitem{lloyd1982least}
Stuart Lloyd.
\newblock Least squares quantization in pcm.
\newblock {\em IEEE transactions on information theory}, 28(2):129--137, 1982.

\bibitem{mccann1997convexity}
Robert~J McCann.
\newblock A convexity principle for interacting gases.
\newblock {\em Advances in mathematics}, 128(1):153--179, 1997.

\bibitem{merigot2016minimal}
Quentin M{\'e}rigot and Jean-Marie Mirebeau.
\newblock Minimal geodesics along volume-preserving maps, through semidiscrete
  optimal transport.
\newblock {\em SIAM Journal on Numerical Analysis}, 54(6):3465--3492, 2016.

\bibitem{pages2015introduction}
Gilles Pag{\`e}s.
\newblock Introduction to vector quantization and its applications for
  numerics.
\newblock {\em ESAIM: proceedings and surveys}, 48:29--79, 2015.

\bibitem{peyre2019computational}
Gabriel Peyr{\'e}, Marco Cuturi, et~al.
\newblock Computational optimal transport: With applications to data science.
\newblock {\em Foundations and Trends{\textregistered} in Machine Learning},
  11(5-6):355--607, 2019.

\bibitem{santambrogio2009absolute}
Filippo Santambrogio.
\newblock Absolute continuity and summability of transport densities: simpler
  proofs and new estimates.
\newblock {\em Calculus of variations and partial differential equations},
  36(3):343--354, 2009.

\bibitem{santambrogio2015optimal}
Filippo Santambrogio.
\newblock {\em Optimal transport for applied mathematicians}, volume~55.
\newblock Springer, 2015.

\bibitem{villani2003topics}
C{\'e}dric Villani.
\newblock {\em Topics in optimal transportation}.
\newblock Number~58. American Mathematical Soc., 2003.

\bibitem{villani2008optimal}
C{\'e}dric Villani.
\newblock {\em Optimal transport: old and new}, volume 338.
\newblock Springer Science \& Business Media, 2008.

\end{thebibliography}
